\documentclass[12pt]{article}
\usepackage[]{amsmath,amssymb}
\usepackage{amscd}
\usepackage{latexsym}
\usepackage{cite}
\usepackage{amsthm}

\usepackage{amsfonts,  eucal} 
\usepackage[dvips]{pict2e}
\usepackage[dvipdfmx]{graphics}

\usepackage{cite}

\newtheorem{definition}{Definition}[section]
\newtheorem{theorem}[definition]{Theorem}
\newtheorem{lemma}[definition]{Lemma}
\newtheorem{corollary}[definition]{Corollary}
\newtheorem{remark}[definition]{Remark}

\newtheorem{conjecture}[definition]{Conjecture}
\newtheorem{problem}[definition]{Problem}
\newtheorem{note}[definition]{Note}

\newtheorem{proposition}[definition]{Proposition}

\begin{document}
\title{\bf Distance-regular graphs, the subconstituent \\algebra, 
and the $Q$-polynomial property
}
\author{
Paul Terwilliger 
}
\date{}

\maketitle
\begin{abstract} 
This survey paper contains a tutorial introduction to distance-regular graphs, with an emphasis on the subconstituent algebra and the  $Q$-polynomial property.
\bigskip


 \end{abstract}
\bigskip 
\def\leaderfill{\leaders\hbox to 1em{\hss.\hss}\hfill}
{\narrower {\narrower {\narrower  \small
		   \centerline{CONTENTS}

\medskip\noindent
 1. Introduction \leaderfill 2 \\
 2. Distance-regular graphs                                  \leaderfill 3  \\
  3. Some polynomials                                 \leaderfill 6  \\
    4. The geometry of the eigenspaces                                \leaderfill 10  \\
 5.  The Krein parameters and the dual distance matrices                    \leaderfill  11 \\
  6. Reduction rules                       \leaderfill  15 \\
 7. The primary $T$-module                      \leaderfill  19 \\
 8. The Krein condition and the triple product relations                     \leaderfill  22 \\
 9. The function algebra and the Norton algebra                    \leaderfill  25 \\
 10. The function algebra and nondegenerate primitive idempotents                   \leaderfill  27 \\
 11. The $Q$-polynomial property and Askey-Wilson duality                  \leaderfill  29 \\
 12.  The function algebra characterization of the $Q$-polynomial property                \leaderfill  33 \\
 13. Irreducible $T$-modules and tridiagonal pairs \leaderfill 34 \\
 14. Recurrent sequences                      \leaderfill  36 \\
 15. The tridiagonal relations                  \leaderfill  38 \\
 16. The primary $T$-module and the Askey-Wilson relations                      \leaderfill  42 \\
17. The Pascasio characterization of the $Q$-polynomial property                     \leaderfill  45 \\
18. Distance-regular graphs with classical parameters                   \leaderfill  48 \\
19. The balanced set characterization of the $Q$-polynomial property                     \leaderfill  49 \\
  20. Directions for future research                   \leaderfill 53\\
References                       \leaderfill 54 \\

 \baselineskip=\normalbaselineskip \par}  \par} \par}

\section{Introduction}

This survey paper contains a tutorial introduction to distance-regular graphs, with an emphasis on the subconstituent algebra and the  $Q$-polynomial property.
Our treatment is roughly based on the unpublished lecture notes \cite{suz}, along with two more recent versions \cite{TerNotes}.
 The treatment is not comprehensive;
instead we restrict out attention to those topics that seem most important. A proof is given for every main result in the paper.
 We do not assume any prior knowledge about distance-regular graphs. We intend the present paper to complement the excellent recent works
 \cite{bbit, dkt}. 
 \medskip
 
\noindent A hypercube is an elementary example of a $Q$-polynomial distance-regular graph. 
The subconstituent algebra of a hypercube is described in \cite{go}.  We advise the
beginning reader to treat the hypercubes as a running example, using \cite{go} as a guide.
\medskip

\noindent As we go along, we will encounter a linear-algebraic object called a tridiagonal pair. As a warmup, we now
define a tridiagonal pair. 
\medskip

\noindent  Let $\mathbb F$ denote a field. Let $V$ denote a nonzero, finite-dimensional vector space over $\mathbb F$.
An $\mathbb F$-linear map $A:V\to V$ is called {\it diagonalizable} whenever
$V$ is spanned by the eigenspaces of $A$.

\begin{definition}  {\rm (See
\cite[Definition 1.1]{itt}.)}
\label{def:tdp}
\rm
A  {\it tridiagonal pair}  on $V$
is an ordered pair of $\mathbb F$-linear maps 
$A:V\to V$ and $A^*:V\to V$  that satisfy
the following four conditions.
\begin{enumerate}
\item[\rm (i)]  Each of $A,A^*$ is diagonalizable.
\item[\rm (ii)] There exists an ordering $\lbrace V_i \rbrace_{i=0}^d $ of the  
eigenspaces of $A$ such that 
\begin{equation}
A^* V_i \subseteq V_{i-1} + V_i+ V_{i+1} \qquad \qquad (0 \leq i \leq d),
\label{eq:t1}
\end{equation}
where $V_{-1} = 0$ and $V_{d+1}= 0$.
\item[\rm (iii)]There exists an ordering $\lbrace V^*_i \rbrace_{i=0}^\delta$ of
the  
eigenspaces of $A^*$ such that 
\begin{equation}
A V^*_i \subseteq V^*_{i-1} + V^*_i+ V^*_{i+1} \qquad \qquad (0 \leq i \leq \delta),
\label{eq:t2}
\end{equation}
where $V^*_{-1} = 0$ and $V^*_{\delta+1}= 0$.
\item[\rm (iv)]  There does not exist a subspace $W$ of $V$ such  that $AW\subseteq W$,
$A^*W\subseteq W$, $W\not=0$, $W\not=V$.
\end{enumerate}
\end{definition}

\begin{note} \rm \label{note:conv}
According to a common 
notational convention, $A^*$ denotes the conjugate transpose
of $A$.
 We are not using this convention.
In a tridiagonal pair $A,A^*$ the maps
$A$ and $A^*$ are arbitrary subject to (i)--(iv) above.
\end{note}

\noindent Referring to Definition \ref{def:tdp}, by \cite[Lemma~4.5]{itt} the integers $d$ and $\delta$ from \eqref{eq:t1}, \eqref{eq:t2} 
are equal; we call this common value the {\it diameter}
of the pair.
\medskip

\noindent  The concept of a tridiagonal pair was formally introduced in \cite{itt}. However, the concept shows up earlier \cite{terwSub1, terwSub2, terwSub3}
in connection with the subconstituent algebra $T$ of a $Q$-polynomial distance-regular graph. 
As we will see, for such a graph every irreducible $T$-module
gives a tridiagonal pair in a natural way.
\medskip

\noindent We refer the reader to   \cite{bbit, itoOq, augIto,IKT, nomTB} for background and historical remarks about tridiagonal pairs.
\medskip

\noindent We will encounter tridiagonal pairs of the following sort.

\begin{definition}\rm \label{def:lp} (See \cite[Definition~1.1]{LS99}.)
A {\it Leonard pair} on $V$ is a tridiagonal pair $A, A^*$ on $V$
such that the eigenspaces $\lbrace V_i\rbrace_{i=0}^d$ and $\lbrace V^*_i\rbrace_{i=0}^d$ all have dimension one.
\end{definition}

\noindent We refer the reader to \cite{bbit, LS99, aa, notesLS} for background and historical remarks about Leonard pairs.

\section{Distance-regular graphs}
We now turn our attention to distance-regular graphs.
After a brief review of the basic definitions,
we will describe the Bose-Mesner algebra, the dual Bose-Mesner algebra,  and the subconstituent
algebra.
For more information we refer the reader to 
\cite{bannai, bbit, bcn, dkt, go,terwSub1,terwSub2, terwSub3}. 

\medskip
\noindent 
Let $\mathbb R$ denote the field of real numbers. 
Let $X$ denote a nonempty  finite  set.
Let ${\rm Mat}_X(\mathbb R)$
denote the $\mathbb R$-algebra
consisting of the matrices with rows and columns  indexed by $X$
and all entries in $\mathbb R  $. Let $I \in{ \rm Mat}_X(\mathbb R)$ denote the identity matrix.
 Let
$V=\mathbb R^X$ denote the vector space over $\mathbb R$
consisting of the column vectors with
coordinates indexed by $X$ and all entries 
in $\mathbb R$.
The algebra ${\rm Mat}_X(\mathbb R)$
acts on $V$ by left multiplication.
We call $V$ the {\it standard module}.
We endow $V$ with a bilinear form $\langle \, , \, \rangle$ 
that satisfies
$\langle u,v \rangle = u^t v$ for 
$u,v \in V$,
where $t$ denotes transpose. This bilinear form is symmetric. For $u \in V$ we abbreviate $\Vert u \Vert^2 = \langle u, u \rangle $.  Note that $\Vert u \Vert^2 \geq 0$, with equality if and only if $u=0$.
For $u, v \in V$ and $B \in {\rm Mat}_X(\mathbb R)$ we have $\langle Bu, v \rangle = \langle u, B^tv\rangle$.
For all $y \in X,$ define a vector $\hat{y} \in V$ that has $y$-coordinate  $1$ and all other coordinates $0$.
The vectors $\lbrace \hat y \rbrace_{y \in X}$ form an orthonormal basis for $V$. 
For later use, define a matrix $J \in {\rm Mat}_X(\mathbb R)$ that has all entries $1$.
\medskip

\noindent
Let $\Gamma = (X, \mathcal R)$ denote a finite, undirected, connected graph,
without loops or multiple edges, with vertex set $X$ and 
edge set
$\mathcal R$.   Vertices $y,z$ are {\it adjacent} whenever $y,z$ form an edge.
Let $\partial $ denote the
path-length distance function for $\Gamma $,  and define
$D = \mbox{max}\lbrace \partial(y,z) \vert y,z \in X\rbrace $.  
We call $D$  the {\it diameter} of $\Gamma $. For $y \in X$ and an integer $i\geq 0$ define
$\Gamma_i(y) = \lbrace z \in X\vert \partial(y,z)=i \rbrace$. We abbreviate $\Gamma(y) = \Gamma_1(y)$.
For an integer $k\geq 0$ we say that $\Gamma$ is {\it regular with
valency $k$} whenever $\vert \Gamma(y)\vert = k$ for all $y \in X$.
 We say that $\Gamma$ is {\it distance-regular}
whenever for all integers $h,i,j\;(0 \le h,i,j \le D)$ 
and for all
vertices $y,z \in X$ with $\partial(y,z)=h,$ the number
$p^h_{i,j} = \vert \Gamma_i(y) \cap \Gamma_j(z) \vert$
is independent of $y$ and $z$. The $p^h_{i,j}$  are called
the {\it intersection numbers} of $\Gamma$. From now until the end of Section 19,  we assume that $\Gamma$ is distance-regular with $D\geq 3$.
By construction $p^h_{i,j} = p^h_{j,i} $ for $0 \leq h,i,j\leq D$.  
We abbreviate
\begin{align*}
c_i = p^i_{1,i-1} \;\; (1 \leq i \leq D), \qquad a_i = p^i_{1,i} \;\; (0 \leq i \leq D), \qquad  b_i = p^i_{1,i+1} \;\; (0 \leq i \leq D-1).
\end{align*}
\noindent  Note that $a_0=0$ and $c_1=1$. Moreover
\begin{align*}
c_i > 0 \quad (1 \leq i \leq D), \qquad \qquad b_i > 0 \quad (0 \leq i \leq D-1).
\end{align*}
The graph $\Gamma$ is regular with valency $k=b_0$. Moreover,
\begin{align*}
c_i + a_i + b_i = k \qquad \qquad (0 \leq i \leq D),
\end{align*}
where $c_0=0$ and $b_D=0$. For $0 \leq i \leq D$ define $k_i = p^0_{i,i}$ and note that $k_i = \vert \Gamma_i(y) \vert$ for all $y \in X$.
We have $k_0=1$ and $k_1=k$.
By a routine counting argument, $k_{i-1} b_{i-1}= k_i c_i $ for $1 \leq i \leq D$. Consequently
\begin{align} \label{eq:kiform}
k_i = \frac{ b_0 b_1 \cdots b_{i-1}}{c_1 c_2 \cdots c_i} \qquad \qquad (0 \leq i \leq D).
\end{align}
\noindent 
By the triangle inequality, the following hold for $0 \leq h,i,j\leq D$:
\begin{enumerate}
\item[\rm (i)] $p^h_{i,j}= 0$ if one of $h,i,j$ is greater than the sum of the other two;
\item[\rm (ii)] $p^h_{i,j}\not=0$ if one of $h,i,j$ is equal to the sum of the other two.
\end{enumerate}
The following results are verified by routine  counting arguments:
\begin{align*}
&p^h_{0,j} = \delta_{h,j} \qquad (0 \leq h,j\leq D); \qquad \qquad 
p^h_{i,0} = \delta_{h,i} \qquad (0 \leq h,i\leq D);\\
&p^0_{i,j} = \delta_{i,j}k_i \qquad (0 \leq i,j\leq D); \qquad \qquad
\sum_{i=0}^D p^h_{i,j} = k_j\qquad (0 \leq h,j\leq D).
\end{align*}
\noindent 
We recall the Bose-Mesner algebra of $\Gamma.$ 
For 
$0 \leq i \leq D$ define a matrix $A_i \in {\rm Mat}_X(\mathbb R)$ with
$(y,z)$-entry
\begin{align*}
(A_i)_{y,z} = \begin{cases}  
1, & {\mbox{\rm if $\partial(y,z)=i$}};\\
0, & {\mbox{\rm if $\partial(y,z) \ne i$}}
\end{cases}
 \qquad (y,z \in X).
\end{align*}
We call $A_i$ the $i^{th}$ {\it distance matrix} of $\Gamma$. 
For $y \in X$ 
we have
$A_i {\hat y} = \sum_{z \in \Gamma_i(y)} {\hat z}$.
We abbreviate $A=A_1$ and call this  the {\it adjacency
matrix} of $\Gamma.$ We observe
(i) $A_0 = I$;
 (ii)
$\sum_{i=0}^D A_i = J$;
(iii)  $A_i^t = A_i  \;(0 \leq i \leq D)$;
(iv) $A_iA_j = \sum_{h=0}^D p^h_{i,j} A_h \;( 0 \leq i,j \leq D) $.
Consequently the matrices
 $\lbrace A_i\rbrace_{i=0}^D$
form a basis for a commutative subalgebra $M$ of ${\rm Mat}_X(\mathbb R)$, called the 
{\it Bose-Mesner algebra} of $\Gamma$.
The distance matrices are symmetric and mutually commute. Therefore they can be simultaneously diagonalized. Consequently
$M$ has a second basis 
$\lbrace E_i\rbrace_{i=0}^D$ such that
(i) $E_0 = |X|^{-1}J$;
(ii) $\sum_{i=0}^D E_i = I$;
(iii) $E_i^t =E_i  \;(0 \le i \le D)$;
(iv) $E_iE_j =\delta_{i,j}E_i  \;(0 \leq i,j \leq D)$.
We call $\lbrace E_i\rbrace_{i=0}^D$  the {\it primitive idempotents}
of $\Gamma$.  The primitive idempotent $E_0$ is called {\it trivial}. We have
\begin{align*}
V = \sum_{i=0}^D E_iV \qquad \qquad (\mbox{\rm orthogonal direct sum}).
\end{align*}
For $0 \leq i \leq D$ the subspace $E_iV$ is a common eigenspace of $M$. 
Note that
\begin{align*}
E_0V = \mathbb R {\bf 1}, \qquad \qquad {\bf 1} = \sum_{y \in X} {\hat y}.
\end{align*}
\noindent
For $0 \leq i \leq D$ let $m_i$ denote the dimension of $E_iV$. We have $m_i = {\rm tr}(E_i)$, where tr denotes trace.
Note that $m_0=1$.
\medskip


\noindent
We  recall the dual Bose-Mesner algebras of $\Gamma$.
For the rest of this section, fix
a vertex $x \in X$. 
For 
$ 0 \leq i \leq D$ let $E_i^*=E_i^*(x)$ denote the diagonal
matrix in 
 ${\rm Mat}_X(\mathbb R)$
 with $(y,y)$-entry
\begin{equation}\label{DEFDEI}
(E_i^*)_{y,y} = \begin{cases} 1, & \mbox{\rm if $\partial(x,y)=i$};\\
0, & \mbox{\rm if $\partial(x,y) \ne i$}
\end{cases}
 \qquad (y \in X).
\end{equation}
We call $E_i^*$ the  $i^{th}$ {\it dual primitive idempotent of $\Gamma$
 with respect to $x$} \cite[p.~378]{terwSub1}. For $y \in X $ we have $E^*_i {\hat y} = {\hat y}$ if $\partial(x,y) =i$, and $E^*_i {\hat y} =0$ if
 $\partial(x,y) \not=i$.
We observe
(i) $\sum_{i=0}^D E_i^*=I$;
(ii) $E_i^{*t} = E_i^*$ $(0 \le i \le D)$;
(iii) $E_i^*E_j^* = \delta_{i,j}E_i^* $ $(0 \le i,j \le D)$.
By these facts 
$\lbrace E_i^*\rbrace_{i=0}^D$ form a 
basis for a commutative subalgebra
$M^*=M^*(x)$ of 
${\rm Mat}_X(\mathbb R)$.
We call 
$M^*$ the {\it dual Bose-Mesner algebra of
$\Gamma$ with respect to $x$} \cite[p.~378]{terwSub1}.

\medskip
\noindent 
We recall the subconstituents of $\Gamma $ with respect to $x$.
From
\eqref{DEFDEI} we find
\begin{equation}\label{DEIV}
E_i^*V = \mbox{\rm Span}\lbrace \hat{y} \vert y \in \Gamma_i(x)\rbrace
\qquad \qquad (0 \leq i \leq D).
\end{equation}
By 
\eqref{DEIV}  and since $\lbrace {\hat y} \rbrace_{y \in X}$
is an orthonormal basis for $V$,
 we find
\begin{align*}
V = \sum_{i=0}^D E^*_iV \qquad \qquad (\mbox{\rm orthogonal direct sum}).
\end{align*}
For $0 \leq i \leq D$ the subspace $E^*_iV$ is a common eigenspace for $M^*$. 
Observe that the dimension of $E^*_iV$ is equal to $k_i$. Also ${\rm tr}(E^*_i)=k_i$.
We call $E_i^*V$ the $i^{th}$ {\it subconstituent of $\Gamma$
with respect to $x$}.  Note that $E^*_0 V = \mathbb R {\hat x}$. Also note that
\begin{align*}
A E^*_iV \subseteq E^*_{i-1}V + E^*_iV + E^*_{i+1}V \qquad \qquad (0 \leq i \leq D),
\end{align*}
where $E^*_{-1}=0$ and $E^*_{D+1}=0$.
\medskip

\noindent
We recall the subconstituent algebra of $\Gamma $ with respect to $x$.
Let $T=T(x)$ denote the subalgebra of ${\rm Mat}_X(\mathbb R)$ generated by 
$M$ and $M^*$. 
Observe that $T$ has finite dimension. The algebra $T$ is closed under the
transpose map, because $M$ and $M^*$ are closed under the transpose map. We call $T$ the {\it subconstituent algebra}
(or {\it Terwilliger algebra}) {\it of $\Gamma$ 
 with respect to $x$} \cite[Definition 3.3]{terwSub1}.
See
\cite{curtin1,
curtin2,
curtin6,
egge1,
go,
hobart,
tanabe,
terwSub1,
terwSub2,
terwSub3}
for more information on the subconstituent
algebra.

\medskip
\noindent We recall the $T$-modules.
By a {\it T-module}
we mean a subspace $W \subseteq V$ such that $BW \subseteq W$
for all $B \in T$. A $T$-module $W$ is called {\it irreducible} whenever $W\not=0$ and $W$ does not contain a $T$-module besides 0 and $W$.
Let $W$ denote a $T$-module and let 
$W'$ denote a  
$T$-module contained in $W$.
Then the orthogonal complement of $W'$ in $W$ is a $T$-module.
It follows that each $T$-module
is an orthogonal direct sum of irreducible $T$-modules.
In particular, $V$ is an orthogonal direct sum of irreducible $T$-modules.
\medskip

\noindent The following relations are verified by matrix multiplication:
\begin{align*}
\vert X \vert E^*_0 E_0 E^*_0 = E^*_0, \qquad \qquad \vert X \vert E_0 E^*_0 E_0 = E_0.
\end{align*}

\section{Some polynomials}
\noindent Throughout this section $\Gamma=(X,\mathcal R)$ denotes a distance-regular graph
with diameter $D\geq 3$. We will discuss two sequences of polynomials that are associated with the distance matrices of $\Gamma$.

\begin{lemma} \label{lem:A3t} We have
\begin{align*}
A A_i &= b_{i-1} A_{i-1} + a_i A_i + c_{i+1} A_{i+1} \qquad \qquad (1 \leq i \leq D-1),\\
A A_D &= b_{D-1} A_{D-1} + a_D A_D.
\end{align*}
\end{lemma} \begin{proof}
This is $A_i A_j = \sum_{h=0}^D p^h_{i,j} A_h$ with $j=1$.
\end{proof}

\noindent Let $\lambda$ denote an indeterminate. Let $\mathbb R\lbrack \lambda \rbrack$ denote the $\mathbb R$-algebra of  polynomials in $\lambda$ that
have all coefficients in $\mathbb R$.

\begin{definition}\label{def:vi} \rm We define some polynomials $\lbrace v_i \rbrace_{i=0}^{D+1} $ in $\mathbb R\lbrack \lambda \rbrack$ such that
\begin{align*}
&v_0=1, \qquad \quad v_1 = \lambda, \\
\lambda v_i &= b_{i-1} v_{i-1} + a_i v_i + c_{i+1} v_{i+1} \qquad \qquad (1 \leq i \leq D),
\end{align*}
where $c_{D+1}=1$.
\end{definition}

\begin{lemma} \label{lem:vi} The following {\rm (i)--(iv)} hold:
\begin{enumerate}
\item[\rm (i)] ${\rm deg} \,v_i = i \quad (0 \leq i \leq D+1)$;
\item[\rm (ii)] the coefficient of $\lambda^i$ in $v_i$ is $(c_1 c_2 \cdots c_i)^{-1} \quad (0 \leq i \leq D+1)$;
\item[\rm (iii)] $v_i(A)=A_i \quad (0 \leq i \leq D)$;
\item[\rm (iv)] $v_{D+1}(A)=0$.
\end{enumerate}
\end{lemma}
\begin{proof} (i), (ii) By Definition \ref{def:vi}. \\
\noindent (iii), (iv) Compare Lemma \ref{lem:A3t} and Definition \ref{def:vi}.
\end{proof}

\begin{corollary} \label{lem:AGen}  The following hold:
\begin{enumerate}
\item[\rm (i)]  the algebra $M$ is generated by $A$;
\item[\rm (ii)]  the minimal polynomial of $A$ is
$c_1 c_2 \cdots c_D v_{D+1}$.
\end{enumerate}
\end{corollary}
\begin{proof} By Lemma \ref{lem:vi} and since $\lbrace A_i\rbrace_{i=0}^D$ is a basis for $M$.
\end{proof}

\noindent Next we consider the eigenvalues of $A$. Since 
 $\lbrace E_i\rbrace_{i=0}^D$ form a basis for  
$M$, there exist real numbers  
$\lbrace\theta_i\rbrace_{i=0}^D$
such that
\begin{align} \label{eq:th}
A = \sum_{i=0}^D \theta_iE_i.
\end{align}

\begin{lemma} \label{lem:thDist} The following {\rm (i)--(iii)} hold:
\begin{enumerate}
\item[\rm (i)] the polynomial $v_{D+1}$ has $D+1$ mutually distinct roots $\lbrace \theta_i \rbrace_{i=0}^D$;
\item[\rm (ii)] the eigenspaces of $A$ are $\lbrace E_iV \rbrace_{i=0}^D$;
\item[\rm (iii)] for $0 \leq i \leq D$, $\theta_i$ is the eigenvalue of $A$ for $E_iV$.
\end{enumerate}
\end{lemma}
\begin{proof} (i) The roots of $v_{D+1}$ are mutually distinct by Corollary \ref{lem:AGen}(ii) and since $A$ is diagonalizable.
These roots are  $\lbrace \theta_i \rbrace_{i=0}^D$ by \eqref{eq:th} .
\\
\noindent (ii), (iii) By  \eqref{eq:th}.
\end{proof}

\begin{definition}\rm 
\label{def:Eig} For $0 \leq i \leq D$ we call $\theta_i$ the {\it $i^{th}$ eigenvalue of $\Gamma$} (with respect to the given ordering of the primitive idempotents).
\end{definition}

\noindent For convenience we adjust the normalization of the polynomials $v_i$.
\begin{definition}\rm \label{def:ui}
Define the polynomial
\begin{align} \label{eq:ui}
u_i = \frac{v_i}{k_i} \qquad \qquad (0 \leq i \leq D).
\end{align}
\end{definition}

\begin{lemma} \label{lem:urec} We have
\begin{align*} 
& u_0 = 1, \qquad \quad u_1= k^{-1} \lambda, \\
 \lambda u_i &= c_i u_{i-1} + a_i u_i + b_i u_{i+1} \qquad \qquad (1 \leq i \leq D-1),\\
 & \lambda u_D -c_D u_{D-1} - a_D u_D = k^{-1}_D v_{D+1}.
\end{align*}
\end{lemma}
\begin{proof} Evaluate the recurrence in Definition \ref{def:vi}
using $v_i = k_i u_i$ $(0 \leq i \leq D)$ and \eqref{eq:kiform}.
\end{proof}

\noindent Recall that $\lbrace A_i \rbrace_{i=0}^D$ and $\lbrace E_i \rbrace_{i=0}^D$ are bases for the vector space $M$. Next we
describe how these bases are related.

\begin{lemma} \label{lem:transAE}
For $0 \leq j \leq D$ we have
\begin{enumerate}
\item[\rm (i)] 
$A_j =  \sum_{i=0}^D v_j(\theta_i) E_i$;
\item[\rm (ii)]  $E_j = \vert X \vert^{-1} m_j \sum_{i=0}^D u_i(\theta_j) A_i$.
\end{enumerate}
\end{lemma}
\begin{proof} (i) We have
\begin{align*}
A_j = v_j(A) = v_j(A) \sum_{i=0}^D E_i = \sum_{i=0}^D v_j(\theta_i) E_i.
\end{align*}
\noindent (ii) Define $S=\vert X \vert^{-1} m_j \sum_{i=0}^D u_i(\theta_j) A_i$. We show that $E_j=S$. Expanding $AS$ using Lemma \ref{lem:A3t}, we routinely obtain
$AS=\theta_j S$. By this and since $S \in M$, we obtain $S=\alpha E_j$ for some $\alpha \in \mathbb R$. In the equation $S=\alpha E_j$,
 take the trace of each side. We have ${\rm tr}(S)=m_j$, because ${\rm tr} (A_\ell) = \delta_{0,\ell} \vert X \vert $ for $0 \leq \ell \leq D$.
 We have ${\rm tr} (E_j)=m_j$. By these comments $\alpha=1$, so $E_j=S$.
\end{proof}

\begin{lemma} \label{lem:AiEi} For $0 \leq i,j\leq D$,
\begin{align*}
E_i A_j  = v_j(\theta_i) E_i = A_j E_i.
\end{align*}
\end{lemma}
\begin{proof} To verify this equation, eliminate $A_j$ using Lemma \ref{lem:transAE}(i) and simplify the result.
\end{proof}

\begin{lemma} \label{lem:norm} For $0 \leq i \leq D$ we have
\begin{enumerate}
\item[\rm (i)]  $v_i(\theta_0)=k_i$;
\item[\rm (ii)] $u_i(\theta_0)=1$.
\end{enumerate}
\end{lemma}
\begin{proof} (i) The matrix $A_i$ has constant row sum $k_i$. Therefore $A_i J =k_i J$. We have $E_0=\vert X \vert^{-1}J$, so
$A_i E_0 =k_i E_0$. By Lemma \ref{lem:AiEi}, $A_i E_0 = v_i(\theta_0) E_0$. By these comments $v_i(\theta_0)=k_i$.
\\
\noindent (ii) By (i) and $v_i = k_i u_i$.
\end{proof}

\noindent It is often said that the polynomials $\lbrace u_i \rbrace_{i=0}^D$ and $\lbrace v_i \rbrace_{i=0}^D$ are orthogonal \cite[p.~201]{bannai}. Our next goal is to explain
what this means. We will bring in a bilinear form, and explain what is orthogonal to what.
\medskip

\noindent We endow the vector space ${\rm Mat}_X(\mathbb R)$ with a bilinear form $\langle\,,\,\rangle$ such that
\begin{align*} 
\langle B, C \rangle = {\rm tr} (BC^t) \qquad \qquad B, C \in {\rm Mat}_X(\mathbb R).
\end{align*}
\noindent This bilinear form is symmetric. For $B \in {\rm Mat}_X(\mathbb R)$ we abbreviate $\Vert B \Vert^2 = \langle B,B \rangle $.  Note that $\Vert B \Vert^2 \geq 0$, with equality if and only if $B=0$. 
\begin{lemma} \label{lem:BCR} For $G,H,K \in {\rm Mat}_X(\mathbb R)$ we have
\begin{align*}
\langle GH, K \rangle = \langle H, G^t K \rangle = \langle G, KH^t \rangle.
\end{align*}
\end{lemma}
\begin{proof} Use ${\rm tr} (BC) = {\rm tr} (CB)$.
\end{proof}

\begin{lemma} \label{lem:bAE}
For $0 \leq  i,j\leq D$ we have
\begin{enumerate}
\item[\rm (i)] $\langle E_i, E_j \rangle = \delta_{i,j} m_i$;
\item[\rm (ii)]  $\langle A_i, A_j \rangle = \delta_{i,j} k_i \vert X \vert$;
\item[\rm (iii)]  $\langle A_i, E_j \rangle = v_i(\theta_j)m_j$.
\end{enumerate}
\end{lemma}
\begin{proof} (i) We have
\begin{align*}
\langle E_i, E_j \rangle = 
{\rm tr}(E_i E^t_j) =
{\rm tr}(E_i E_j) =
\delta_{i,j} {\rm tr}(E_i) =
\delta_{i,j} m_i.
\end{align*}
\noindent (ii) We have
\begin{align*}
\langle A_i, A_j \rangle = {\rm tr}(A_i A^t_j) = {\rm tr}(A_iA_j) = \sum_{h=0}^D p^h_{i,j} {\rm tr}(A_h) = p^0_{i,j} \vert X \vert = \delta_{i,j} k_i \vert X \vert.
\end{align*}
\noindent (iii) Eliminate $A_i$ using Lemma \ref{lem:transAE}(i), and evaluate the result using (i) above.
\end{proof}

\begin{proposition} \label{prop:uorth} We have 
\begin{align*}
\sum_{\ell=0}^D u_\ell(\theta_i) u_\ell (\theta_j) k_\ell &= \delta_{i,j} m^{-1}_i \vert X \vert \qquad \qquad (0 \leq i,j\leq D),
\\
\sum_{\ell=0}^D u_i(\theta_\ell) u_j(\theta_\ell ) m_\ell &= \delta_{i,j} k^{-1}_i \vert X \vert \qquad \qquad (0 \leq i,j\leq D).
\end{align*}
\end{proposition}
\begin{proof} The  first equation comes from $\langle E_i, E_j \rangle = \delta_{i,j} m_i$. In this equation, eliminate $E_i$ and $E_j$ using Lemma  \ref{lem:transAE}(ii), 
and evaluate the result using Lemma \ref{lem:bAE}(ii). The second equation comes from $\langle A_i, A_j \rangle = \delta_{i,j} k_i \vert X \vert$. In this equation,
eliminate $A_i$ and $A_j$ using  Lemma  \ref{lem:transAE}(i), and evaluate the result using  Lemma \ref{lem:bAE}(i).
\end{proof}

\begin{proposition} \label{prop:vorth}  We have 
\begin{align*}
\sum_{\ell=0}^D v_\ell(\theta_i) v_\ell(\theta_j) k^{-1}_\ell &= \delta_{i,j} m^{-1}_i \vert X \vert \qquad \qquad (0 \leq i,j\leq D),
\\
\sum_{\ell=0}^D v_i(\theta_\ell) v_j(\theta_\ell) m_\ell &= \delta_{i,j} k_i \vert X \vert \qquad \qquad (0 \leq i,j\leq D).
\end{align*}
\end{proposition}
\begin{proof} Combine Definition  \ref{def:ui}
and 
Proposition \ref{prop:uorth}.
\end{proof}
\begin{note} \rm The relations in Proposition \ref{prop:uorth}  and Proposition \ref{prop:vorth} are called the {\it orthogonality relations}
for the polynomials $\lbrace u_i \rbrace_{i=0}^D$ and $\lbrace v_i\rbrace_{i=0}^D$, respectively.
\end{note}

\noindent Our next goal is to give some formulas for the intersection numbers $p^h_{i,j}$.

\begin{lemma} \label{lem:phij}
For $0 \leq h,i,j\leq D$,
\begin{align*}
p^h_{i,j} = \vert X \vert^{-1} k^{-1}_h \langle A_i A_j, A_h \rangle 
             = \vert X \vert^{-1} k^{-1}_h \langle A_h, A_i A_j \rangle.
             \end{align*}
\end{lemma}
\begin{proof} To verify these equations, expand $A_i A_j$ using $ A_i A_j = \sum_{\ell=0}^D p^\ell_{i,j} A_\ell $, and
evaluate the results using Lemma  \ref{lem:bAE}(ii).
\end{proof}

\begin{lemma} \label{lem:kp}
For $0 \leq h,i,j\leq D$,
\begin{align*}
k_h p^h_{i,j} = k_i p^i_{j,h} = k_j p^j_{h,i} = \vert X \vert^{-1} {\rm tr}\bigl( A_h A_i A_j\bigr).
\end{align*}
\end{lemma}
\begin{proof} Routine application of Lemma \ref{lem:phij}.
\end{proof} 


\begin{proposition} \label{prop:inter}
For $0 \leq h,i,j\leq D$,
\begin{align*}
p^h_{i,j} = \vert X \vert^{-1} k_i k_j \sum_{\ell=0}^D u_i(\theta_\ell) u_j(\theta_\ell) u_h(\theta_\ell)m_\ell.
\end{align*}
\end{proposition}
\begin{proof} In the equation $p^h_{i,j} =  \vert X \vert^{-1} k^{-1}_h \langle A_i A_j, A_h \rangle $, eliminate $A_h, A_i, A_j$ using Lemma  \ref{lem:transAE}(i),
 and evaluate the result using Lemma  \ref{lem:bAE}(i).
\end{proof}

\section{The geometry of the eigenspaces}
\noindent Throughout this section $\Gamma=(X,\mathcal R)$ denotes a distance-regular graph
with diameter $D\geq 3$. Recall the standard module $V$ and the adjacency matrix $A$. Recall that for $0 \leq j \leq D$ the subspace $E_jV$ 
is an eigenspace of $A$ with eigenvalue $\theta_j$. This eigenspace is spanned by the vectors $\lbrace E_j {\hat w} \vert w \in X\rbrace$. 
Note that for $y,z\in X$ the following scalars are equal:
\begin{align*}
&{\mbox{ $\langle E_j {\hat y}, E_j {\hat z}\rangle$}}, \qquad \qquad \qquad 
{\mbox{ $(y,z)$-entry of $E_j$}}, \\
&
{\mbox{ $\langle {\hat y}, E_j {\hat z}\rangle$}}, \qquad \qquad \qquad 
{\mbox{ $y$-coordinate of  $E_j {\hat z}$}}, \\
&
{\mbox{ $\langle E_j {\hat y}, {\hat z}\rangle$}},
\qquad \qquad \qquad 
{\mbox{ $z$-coordinate of $E_j {\hat y}$}}.
\end{align*}
\noindent Next, we have some comments of a geometric nature.
\begin{lemma} \label{lem:uij} For $0 \leq i,j\leq D$ and $y,z \in X$ with $\partial(y,z)=i$,
\begin{enumerate}
\item[\rm (i)] $\langle E_j {\hat y}, E_j {\hat z} \rangle = \vert X \vert^{-1} m_j u_i(\theta_j)$;
\item[\rm (ii)] $\Vert E_j {\hat y} \Vert^2 = \Vert E_j {\hat z} \Vert^2 = \vert X \vert^{-1} m_j$;
\item[\rm (iii)] $\displaystyle{u_i(\theta_j) = \frac{\langle E_j {\hat y}, E_j {\hat z}\rangle}{ \Vert E_j {\hat y} \Vert \Vert E_j {\hat z} \Vert}}$;
\item[\rm (iv)] $u_i(\theta_j)$ is the cosine of the angle between $E_j {\hat y}$ and $E_j {\hat z}$.
\end{enumerate}
\end{lemma}
\begin{proof} (i) The $(y,z)$-entry of $E_j$ is found in Lemma  \ref{lem:transAE}(ii).
\\
\noindent (ii) Set $y=z$ and $i=0$ in part (i). \\
\noindent (iii) Combine (i), (ii). \\
\noindent (iv) By (iii) and trigonometry.
\end{proof}

\begin{corollary} \label{cor:Ubound}
We have 
\begin{align*}
  -1 \leq u_i(\theta_j) \leq 1\qquad \qquad  (0 \leq i,j\leq D).
  \end{align*}
  \end{corollary}
  \begin{proof} By Lemma \ref{lem:uij}(iv) and trigonometry.
  \end{proof}

\begin{corollary} \label{cor:dep1} For $0 \leq i,j\leq D$
the following are equivalent:
\begin{enumerate}
\item[\rm (i)] $u_i(\theta_j)=1$;
\item[\rm (ii)] $E_j {\hat y} = E_j{\hat z}$ for all $y,z \in X$ at $\partial(y,z)=i$;
\item[\rm (iii)] there exists $y,z \in X$ such that $\partial(y,z)=i$ and $E_j {\hat y} = E_j {\hat z}$.
\end{enumerate}
\end{corollary}
\begin{proof} By Lemma \ref{lem:uij}(iv) and trigonometry.
\end{proof}

\begin{corollary} \label{cor:dep2} For $0 \leq i,j\leq D$
the following are equivalent:
\begin{enumerate}
\item[\rm (i)] $u_i(\theta_j)=-1$;
\item[\rm (ii)] $E_j {\hat y} = -E_j{\hat z}$ for all $y,z \in X$ at $\partial(y,z)=i$;
\item[\rm (iii)] there exists $y,z \in X$ such that $\partial(y,z)=i$ and $E_j {\hat y} = -E_j {\hat z}$.
\end{enumerate}
\end{corollary}
\begin{proof} By Lemma \ref{lem:uij}(iv) and trigonometry.
\end{proof}

\noindent  The following reformulation of Corollary \ref{cor:dep1} will be useful.
\begin{lemma}\rm \label{def:nondeg} For $0 \leq j \leq D$ the following are equivalent:
\begin{enumerate}
\item[\rm (i)] $u_i(\theta_j) \not=1$ for $1 \leq i \leq D$;
\item[\rm (ii)] the vectors $\lbrace E_j {\hat y} \vert y \in X\rbrace$ are mutually distinct.
\end{enumerate}
\noindent Assume that {\rm (i), (ii)} hold. Then $j \not=0$.
\end{lemma}
\begin{proof} The equivalence of (i), (ii) follows from Corollary \ref{cor:dep1}. The last assertion is from Lemma \ref{lem:norm}(ii).
\end{proof}

\begin{definition}\rm For $0 \leq j \leq D$, we call $E_j$  {\it nondegenerate} whenever
the equivalent conditions (i), (ii) hold in Lemma \ref{def:nondeg}. In this case $j\not=0$.
\end{definition}

\noindent The following definition is motivated by Lemma \ref{lem:uij}(iv).
\begin{definition}\rm For $0 \leq j \leq D$, we call the sequence
$\lbrace u_i(\theta_j)\rbrace_{i=0}^D$ the {\it cosine sequence} of $E_j$ (or $\theta_j$).
\end{definition}

\begin{lemma} \label{lem:recognize} {\rm (See \cite[Section~4.1.B]{bcn}.)} For a real number $\theta$ and a sequence of real numbers $\lbrace \sigma_i \rbrace_{i=0}^D$, the following are equivalent:
\begin{enumerate}
\item[\rm (i)] $\theta$ is an eigenvalue of $\Gamma$ with cosine sequence $\lbrace \sigma_i \rbrace_{i=0}^D$;
\item[\rm (ii)] $\sigma_0=1$, $\sigma_1 = k^{-1} \theta$, and
\begin{align*}
c_i \sigma_{i-1} + a_i \sigma_i + b_i \sigma_{i+1} &= \theta \sigma_i \qquad \qquad (1 \leq i \leq D-1),\\
c_D \sigma_{D-1} + a_D \sigma_D &= \theta \sigma_D.
\end{align*}
\end{enumerate}
\end{lemma}
\begin{proof} Use Lemma \ref{lem:thDist}(i) and Lemma \ref{lem:urec}.
\end{proof}

\section{The Krein parameters and  the dual distance matrices}
Our next topic is the Krein parameters. Throughout this section $\Gamma=(X,\mathcal R)$ denotes a distance-regular graph
with diameter $D\geq 3$.
\medskip

\noindent For $B, C \in {\rm Mat}_X(\mathbb R)$ define the matrix $B \circ C \in {\rm Mat}_X(\mathbb R)$ with entries
\begin{align*}
(B \circ C)_{y,z} = B_{y,z} C_{y,z} \qquad \qquad (y,z \in X).
\end{align*}
The operation $\circ$ is called  entrywise multiplication, or Schur multiplication, or Hadamard multiplication.
We have $A_i \circ A_j = \delta_{i,j} A_i$ for $0 \leq i,j\leq D$. Recall that $\lbrace A_i \rbrace_{i=0}^D$ is a basis for the Bose-Mesner algebra $M$.
By these comments, $M$  is closed under $\circ$.
Recall that $\lbrace E_i \rbrace_{i=0}^D$ is a basis for $M$. Therefore, there exist real numbers $q^{h}_{i,j} $ $(0 \leq h,i,j\leq D)$ such that
\begin{align} 
\label{eq:EcE}
E_i \circ E_j = \vert X \vert^{-1} \sum_{h=0}^D q^h_{i,j} E_h \qquad \qquad (0 \leq i,j\leq D).
\end{align} 
The $q^{h}_{i,j}$ are called the {\it Krein parameters} of $\Gamma$. By construction $q^h_{i,j} = q^h_{j,i}$ for $0 \leq h,i,j\leq D$.
Shortly we will show that $q^h_{i,j}\geq 0$ for $0 \leq h,i,j\leq D$.
\medskip

\noindent In order to avoid dealing directly with entrywise multiplication, we bring in a certain map $p$.
For the rest of this section, fix a vertex $x \in X$. 

\begin{definition}\label{def:rho} \rm For $B \in {\rm Mat}_X(\mathbb R)$ let $B^p$ denote the diagonal matrix in ${\rm Mat}_X(\mathbb R)$ with $(y,y)$-entry
\begin{align*}
 (B^p)_{y,y} = B_{x,y} \qquad \qquad (y \in X).
 \end{align*}
 \end{definition}

\begin{lemma} \label{lem:rhoH} For $B, C \in {\rm Mat}_X(\mathbb R)$,
\begin{align*} 
(B \circ C)^p = B^p C^p.
\end{align*}
\end{lemma}
\begin{proof} By Definition \ref{def:rho}.
\end{proof}

\begin{lemma} We have
\begin{enumerate}
\item[\rm (i)] $A_i^p = E^*_i \qquad (0 \leq i \leq D)$;
\item[\rm (ii)] $I^p = E^*_0$;
\item[\rm (iii)] $J^p=I$.
\end{enumerate}
\end{lemma}
\begin{proof} This is routinely checked using 
 Definition \ref{def:rho}.
\end{proof}

\noindent Recall the dual Bose-Mesner algebra $M^*= M^*(x)$.

\begin{lemma} \label{lem:ri} The restriction $p \vert_M : M \to M^*$ is an isomorphism of vector spaces.

\end{lemma}
\begin{proof} The map $p$ is $\mathbb R$-linear, and sends the basis $\lbrace A_i \rbrace_{i=0}^D$ of $M$ to the basis $\lbrace E^*_i\rbrace_{i=0}^D$ of $M^*$.
\end{proof}
\noindent We caution the reader that the map in Lemma \ref{lem:ri} is not an algebra isomorphism in general.

\begin{lemma} \label{lem:rip} For $B, C \in M$,
\begin{align} \label{eq:BCr}
\langle B^p, C^p\rangle = \vert X \vert^{-1} \langle B, C \rangle.
\end{align}
\end{lemma}
\begin{proof} Using Definition \ref{def:rho} one finds  that each side of \eqref{eq:BCr}  is equal to the $(x,x)$-entry of $BC^t$.
\end{proof}

\noindent We mention some useful facts about the map $p$.

\begin{lemma} \label{lem:32} For $B \in M$ we have 
\begin{enumerate}
\item[\rm (i)] 
$E_0 E^*_0 B = E_0 B^p$;
\item[\rm (ii)] 
 $E^*_0 E_0 B^p=
\vert X \vert^{-1} E^*_0 B$; 
\item[\rm (iii)] 
 $ B E^*_0 E_0=  B^p E_0$;
 \item[\rm (iv)]
  $B^p E_0 E^*_0=
\vert X \vert^{-1} B E^*_0$.
\end{enumerate}
\end{lemma} 
\begin{proof} (i) Recall that $E_0 = \vert X \vert^{-1} J$. Recall that for $E^*_0$, the $(x,x)$-entry is 1 and all other entries are 0.
For $y, z \in X$ we compare the $(y,z)$-entry of each side of the given equation. We have
\begin{align*}
(E_0 E^*_0 B)_{y,z} = (E_0)_{y,x} B_{x,z} = \vert X \vert^{-1} B_{x,z}  = (E_0)_{y,z} (B^p)_{z,z} = (E_0 B^p)_{y,z}.
\end{align*}
\\
\noindent (ii) In the equation (i), multiply each side on the left by $E^*_0$ and evaluate the result using $\vert X \vert E^*_0 E_0 E^*_0= E^*_0$.
\\
\noindent (iii), (iv) Take the transpose of each side in (i), (ii) above.
\end{proof}

\begin{definition}\label{def:As} \rm For $0 \leq i \leq D$, define the matrix $A^*_i = A^*_i(x) $ by
\begin{align*}
A^*_i = \vert X \vert (E_i)^p.
\end{align*}
Thus $A^*_i$ is diagonal with $(y,y)$-entry
\begin{align}
 (A^*_i)_{y,y} = \vert X \vert (E_i)_{x,y} \qquad \quad (y \in X).
 \label{eq:As}
 \end{align}
 \noindent We call $A^*_i$ the {\it $i^{\rm th}$ dual distance matrix of $\Gamma$ with respect to $x$}.
 \end{definition}

\begin{lemma} \label{lem:Asbasis} With the above notation,
\begin{enumerate}
\item[\rm (i)] the matrices $\lbrace A^*_i \rbrace_{i=0}^D$ form a basis for $M^*$;
\item[\rm (ii)] for $0 \leq i \leq D$ and $y \in X$,
\begin{align*}
 A^*_i \hat y = m_i u_j(\theta_i) \hat y, \qquad \qquad j = \partial(x,y).
 \end{align*}
 \end{enumerate}
 \end{lemma}
 \begin{proof} (i) By Lemma \ref{lem:ri} and since $\lbrace E_i \rbrace_{i=0}^D$ form a basis for $M$.
 \\
 \noindent (ii) By \eqref{eq:As} and since $\vert X \vert (E_i)_{x,y} = m_i u_j(\theta_i)$.
 \end{proof}

\begin{lemma} \label{lem:list} The following {\rm (i)--(iv)} hold:
\begin{enumerate}
\item[\rm (i)] $A^*_0=I$;
\item[\rm (ii)] $\sum_{i=0}^D A^*_i = \vert X \vert E^*_0$;
\item[\rm (iii)] $(A^*_i)^t = A^*_i \qquad (0 \leq i \leq D)$;
\item[\rm (iv)] $A^*_i A^*_j = \sum_{h=0}^D q^h_{i,j} A^*_h \qquad (0 \leq i,j\leq D)$.
\end{enumerate}
\end{lemma}
\begin{proof} (i) Use $A^*_0 = \vert X \vert E^p_0$ and $E_0 = \vert X \vert^{-1} J$. \\
\noindent (ii) Apply $p$ to each side of $\sum_{i=0}^D E_i = I$. \\
\noindent (iii) $A^*_i$ is diagonal. \\
\noindent (iv) Apply $p$ to each side of \eqref{eq:EcE}, and evaluate the result using Lemma  \ref{lem:rhoH} along with
Definition \ref{def:As}.
\end{proof}

\noindent  We have seen that  $\lbrace E^*_i \rbrace_{i=0}^D$ and $\lbrace A^*_i \rbrace_{i=0}^D$ are bases for the vector space $M^*$. Next we
describe how these bases are related.
\begin{lemma} \label{lem:transAsEs}
For $0 \leq j \leq D$ we have
\begin{enumerate}
\item[\rm (i)] 
$A^*_j = m_j \sum_{i=0}^D u_i(\theta_j) E^*_i$;
\item[\rm (ii)]  $E^*_j = \vert X \vert^{-1} \sum_{i=0}^D v_j(\theta_i) A^*_i$.
\end{enumerate}
\end{lemma}
\begin{proof} In Lemma \ref{lem:transAE} we displayed some equations that show how  $\lbrace E_i\rbrace_{i=0}^D$ and $\lbrace A_i \rbrace_{i=0}^D$ are related.
Apply $p$ to each side of these equations.
\end{proof}
\begin{lemma} \label{lem:AisEis} For $0 \leq i,j\leq D$,
\begin{align*}
E^*_i A^*_j  = m_j u_i(\theta_j) E^*_i = A^*_j E^*_i.
\end{align*}
\end{lemma}
\begin{proof} To verify this equation, eliminate $A^*_j$ using Lemma \ref{lem:transAsEs}(i) and simplify the result.
\end{proof}

\begin{lemma} \label{lem:bAsEs}
For $0 \leq  i,j\leq D$ we have
\begin{enumerate}
\item[\rm (i)] $\langle E^*_i, E^*_j \rangle = \delta_{i,j} k_i$;
\item[\rm (ii)]  $\langle A^*_i, A^*_j \rangle = \delta_{i,j} m_i \vert X \vert$;
\item[\rm (iii)]  $\langle A^*_i, E^*_j \rangle = m_i v_j(\theta_i)$.
\end{enumerate}
\end{lemma}
\begin{proof} (i) Routine.
\\
 (ii) Use Lemma \ref{lem:rip} and Definition \ref{def:As}.\\
\noindent (iii) In the given equation, eliminate $E^*_j$ using Lemma \ref{lem:transAsEs}(ii),
and simplify the result using (ii) above.
\end{proof}

\noindent In the next few lemmas, we describe the Krein parameters in various ways.

\begin{lemma} \label{lem:qhij}
For $0 \leq h,i,j\leq D$,
\begin{align*}
q^h_{i,j} = \vert X \vert^{-1} m^{-1}_h \langle A^*_i A^*_j, A^*_h \rangle 
             = \vert X \vert^{-1} m^{-1}_h \langle A^*_h, A^*_i A^*_j \rangle.
             \end{align*}
\end{lemma}
\begin{proof} To verify these equations, expand $A^*_i A^*_j$ using $ A^*_i A^*_j = \sum_{\ell=0}^D q^\ell_{i,j} A^*_\ell $, and
evaluate the results using Lemma  \ref{lem:bAsEs}(ii).
\end{proof}

\begin{lemma} \label{lem:mq}
For $0 \leq h,i,j\leq D$,
\begin{align*}
m_h q^h_{i,j} = m_i q^i_{j,h} = m_j q^j_{h,i} = \vert X \vert^{-1} {\rm tr}\bigl(A^*_h A^*_i A^*_j\bigr).
\end{align*}
\end{lemma}
\begin{proof} Routine application of Lemma \ref{lem:qhij}.
\end{proof} 


\begin{lemma}
\label{lem:details}
The following {\rm (i)--(iv)} hold.
\begin{enumerate}
\item[\rm (i)] $q^h_{0,j} = \delta_{h,j} \qquad (0 \leq h,j\leq D)$;
\item[\rm (ii)] $q^h_{i,0} = \delta_{h,i} \qquad (0 \leq h,i\leq D)$;
\item[\rm (iii)] $q^0_{i,j} = \delta_{i,j}m_i \qquad (0 \leq i,j\leq D)$;
\item[\rm (iv)] $\sum_{i=0}^D q^h_{i,j} = m_j\qquad (0 \leq h,j\leq D)$.
\end{enumerate}
\end{lemma}
\begin{proof} (i) By Lemmas \ref{lem:bAsEs}, \ref{lem:qhij} we obtain
\begin{align*}
q^h_{0,j} = \vert X \vert^{-1} m^{-1}_h \langle A^*_0 A^*_j , A^*_h\rangle = \vert X \vert^{-1} m^{-1}_h \langle A^*_j , A^*_h\rangle = \delta_{h,j}.
\end{align*}
\noindent (ii) By (i) and $q^h_{i,0}=q^h_{0,i}$. \\
(iii) By (i) and Lemma \ref{lem:mq}. \\
\noindent (iv) We have
\begin{align*}
\sum_{i=0}^D q^h_{i,j} = \vert X \vert^{-1} m^{-1}_h \sum_{i=0}^D \langle A^*_i A^*_j, A^*_h\rangle 
             =  m^{-1}_h  \langle E^*_0A^*_j, A^*_h\rangle 
             =  m^{-1}_h m_j  \langle E^*_0, A^*_h\rangle 
             =m_j.
             \end{align*}         
\end{proof}

\begin{proposition} \label{prop:krein}
For $0 \leq h,i,j\leq D$,
\begin{align*}
q^h_{i,j} = \vert X \vert^{-1} m_i m_j \sum_{\ell=0}^D u_\ell(\theta_i) u_\ell(\theta_j) u_\ell(\theta_h)k_\ell.
\end{align*}
\end{proposition}
\begin{proof} In the equation $q^h_{i,j} =  \vert X \vert^{-1} m^{-1}_h \langle A^*_i A^*_j, A^*_h \rangle $, eliminate $A^*_h, A^*_i, A^*_j$ using Lemma  \ref{lem:transAsEs}(i),
 and evaluate the result using Lemma  \ref{lem:bAsEs}(i).
\end{proof}

\section{Reduction rules}

\noindent Throughout this section $\Gamma=(X,\mathcal R)$ denotes a distance-regular graph with diameter $D\geq 3$. Fix $x \in X$ and write $T=T(x)$.
We will display a number of relations involving $E_0$ and $E^*_0$. These relations are informally known as reduction rules; see \cite[Section~7]{egge1},
\cite[Sections~9, 11, 13]{nomIS}.

\begin{lemma} For $0 \leq i \leq D$,
\begin{align*}
&E_0 E^*_0 A_i = E_0 E^*_i, \qquad \qquad E^*_0 E_0 E^*_i = \vert X \vert^{-1} E^*_0 A_i,
\\
& A_i E^*_0 E_0 = E^*_i E_0, \qquad \qquad E^*_i E_0 E^*_0 = \vert X \vert^{-1} A_i E^*_0.
\end{align*}
\end{lemma}
\begin{proof} Apply Lemma \ref{lem:32} with $B=A_i$.
\end{proof}

\begin{lemma} For $0 \leq i \leq D$,
\begin{align*}
&E^*_0 E_0 A^*_i = E^*_0 E_i, \qquad \qquad E_0 E^*_0 E_i = \vert X \vert^{-1} E_0 A^*_i,
\\
& A^*_i E_0 E^*_0 = E_i E^*_0, \qquad \qquad E_i E^*_0 E_0 = \vert X \vert^{-1} A^*_i E_0.
\end{align*}
\end{lemma}
\begin{proof} Apply Lemma \ref{lem:32} with $B=E_i$.
\end{proof}

\begin{lemma} \label{lem:r1} 
For $0 \leq i,j\leq D$ we have
\begin{enumerate}
\item[\rm (i)] $E_0 A^*_i E_j = \delta_{i,j} E_0 A^*_i$;
\item[\rm (ii)] $E_0 E^*_i E_j = \vert X \vert^{-1} k_i u_i(\theta_j) E_0 A^*_j$;
\item[\rm (iii)] $E_0 A^*_i A_j = k_j u_j(\theta_i) E_0 A^*_i$;
\item[\rm (iv)] $E_0 E^*_i A_j = \sum_{h=0}^D p^h_{i,j} E_0 E^*_h$.
\end{enumerate}
\end{lemma}
\begin{proof} (i) Observe
\begin{align*}
E_0 A^*_i E_j = \vert X \vert E_0 E^*_0 E_i E_j = \delta_{i,j} \vert X \vert E_0 E^*_0 E_i = \delta_{i,j} E_0 A^*_i.
\end{align*}
\noindent (ii) Observe
\begin{align*}
E_0 E^*_i E_j = E_0 E^*_0 A_i E_j = k_i u_i(\theta_j) E_0 E^*_0 E_j = \vert X \vert^{-1} k_i u_i(\theta_j) E_0 A^*_j.
\end{align*}
\noindent (iii) Observe
\begin{align*}
E_0 A^*_i A_j = \vert X \vert E_0 E^*_0 E_i A_j =  \vert X \vert k_j u_j(\theta_i) E_0 E^*_0 E_i   =    k_j u_j(\theta_i) E_0 A^*_i.
\end{align*}
\noindent (iv) Observe
\begin{align*}
E_0 E^*_i A_j =  E_0 E^*_0 A_i A_j =  \sum_{h=0}^D p^h_{i,j} E_0 E^*_0 A_h = \sum_{h=0}^D p^h_{i,j} E_0 E^*_h.
\end{align*}
\end{proof}

\begin{lemma} \label{lem:r2} 
For $0 \leq i,j\leq D$ we have
\begin{enumerate}
\item[\rm (i)] $E^*_0 A_i E^*_j = \delta_{i,j} E^*_0 A_i$;
\item[\rm (ii)] $E^*_0 E_i E^*_j = \vert X \vert^{-1} m_i u_j (\theta_i) E^*_0 A_j$;
\item[\rm (iii)] $E^*_0 A_i A^*_j = m_j u_i(\theta_j) E^*_0 A_i$;
\item[\rm (iv)] $E^*_0 E_i A^*_j = \sum_{h=0}^D q^h_{i,j} E^*_0 E_h$.
\end{enumerate}
\end{lemma}
\begin{proof} Similar to the proof of Lemma \ref{lem:r1}.
\end{proof}

\begin{lemma} \label{lem:r3}
For $0 \leq i,j\leq D$ we have
\begin{enumerate}
\item[\rm (i)]  $E_j A^*_i E_0 = \delta_{i,j} A^*_i E_0$;         
\item[\rm (ii)]  $E_j E^*_i E_0 =  \vert X \vert^{-1} k_i u_i(\theta_j) A^*_jE_0$;
\item[\rm (iii)] $A_j A^*_iE_0 = k_j u_j(\theta_i) A^*_iE_0$;
\item[\rm (iv)] $A_j E^*_i E_0 = \sum_{h=0}^D p^h_{i,j} E^*_hE_0$.
\end{enumerate}
\end{lemma}
\begin{proof} Take the transpose of everything in Lemma \ref{lem:r1}.
\end{proof}

\begin{lemma} \label{lem:r4} 
For $0 \leq i,j\leq D$ we have
\begin{enumerate}
\item[\rm (i)] $E^*_j A_i E^*_0  = \delta_{i,j} A_iE^*_0$;
\item[\rm (ii)] $E^*_j E_i E^*_0 = \vert X \vert^{-1} m_i u_j (\theta_i) A_j E^*_0$;
\item[\rm (iii)] $A^*_j A_i E^*_0  = m_j u_i(\theta_j) A_i E^*_0 $;
\item[\rm (iv)] $A^*_j E_i E^*_0 = \sum_{h=0}^D q^h_{i,j} E_h E^*_0$.
\end{enumerate}
\end{lemma}
\begin{proof} Take the transpose of everything in Lemma \ref{lem:r2}.
\end{proof}

\begin{lemma}
\label{lem:r5}
For $0 \leq i \leq D$ we have
\begin{enumerate}
\item[\rm (i)] $E_0 E^*_i E_0 = \vert X \vert^{-1} k_i E_0$;
\item[\rm (ii)] $E^*_0 E_i E^*_0 = \vert X \vert^{-1} m_i E^*_0$.
\end{enumerate}
\end{lemma}
\begin{proof}(i) Observe
\begin{align*}
E_0 E^*_i E_0 = E_0 E^*_0 A_i E_0 = k_i E_0 E^*_0 E_0 = \vert X \vert^{-1} k_i E_0.
\end{align*}
\noindent (ii) Observe
\begin{align*}
E^*_0 E_i E^*_0 = E^*_0 E_0 A^*_i E^*_0 = m_i E^*_0 E_0 E^*_0 = \vert X \vert^{-1} m_i E^*_0.
\end{align*}
\end{proof}

\begin{lemma} \label{lem:0i0}
For $0 \leq i ,j\leq D$ we have
\begin{enumerate}
\item[\rm (i)] $A_i E^*_0 A_j = \vert X \vert E^*_i E_0 E^*_j$;
\item[\rm (ii)] $E_i E^*_0 A_j = A^*_i E_0 E^*_j$;
\item[\rm (iii)] $A_i E^*_0 E_j= E^*_i E_0 A^*_j$;
\item[\rm (iv)] $E_i E^*_0 E_j = \vert X \vert^{-1} A^*_i E_0 A^*_j$.
\end{enumerate}
\end{lemma}
\begin{proof} (i) Observe
\begin{align*}
A_i E^*_0 A_j = \vert X \vert A_i E^*_0 E_0 E^*_j = \vert X \vert E^*_iE_0 E^*_j.
\end{align*}
\noindent (ii) Observe
\begin{align*}
E_i E^*_0 A_j = \vert X \vert E_i E^*_0 E_0 E^*_j = A^*_i E_0 E^*_j.
\end{align*}
\noindent (iii) Observe
\begin{align*}
A_i E^*_0 E_j = A_i E^*_0 E_0 A^*_j= E^*_i E_0 A^*_j.
\end{align*}
\noindent (iv) Observe
\begin{align*}
E_i E^*_0 E_j = E_i E^*_0 E_0 A^*_j = \vert X \vert^{-1} A^*_i E_0 A^*_j.
\end{align*}
\end{proof} 

\begin{corollary} $ME^*_0 M$ and $M^* E_0 M^*$ span the same subspace of ${\rm Mat}_X(\mathbb R)$.
\end{corollary}
\begin{proof}  By Lemma \ref{lem:0i0}.
\end{proof}

\begin{lemma} {\rm (See \cite[Proposition~11.1]{egge1}.)} We have
\begin{align}
\label{eq:km}
\sum_{i=0}^D k^{-1}_i E^*_i E_0 E^*_i  = \sum_{j=0}^D m^{-1}_j E_j E^*_0 E_j.
\end{align}
\end{lemma}
\begin{proof} Observe
\begin{align*}
\sum_{i=0}^D k^{-1}_i E^*_i E_0 E^*_i 
&= \vert X \vert^{-1} \sum_{i=0}^D k^{-1}_i A_i E^*_0 A_i \\
&= \vert X \vert^{-1} \sum_{i=0}^D k^{-1}_i \Biggl( \sum_{r=0}^D v_i(\theta_r) E_r\Biggr) E^*_0 \Biggl(\sum_{s=0}^D v_i(\theta_s)E_s\Biggr)
\\
&= \vert X \vert^{-1} \sum_{r=0}^D \sum_{s=0}^D E_r E^*_0 E_s \Biggl( \sum_{i=0}^D k^{-1}_i v_i(\theta_r) v_i(\theta_s)\Biggr)
\\
&= \vert X \vert^{-1} \sum_{r=0}^D \sum_{s=0}^D E_r E^*_0 E_s \Bigl( \delta_{r,s} m^{-1}_r \vert X\vert\Bigr)        
\\
&= \sum_{j=0}^D m^{-1}_j E_j E^*_0 E_j.
\end{align*}
\end{proof}

\begin{definition} \label{def:e0} \rm Define the matrix $e_0 = e_0(x) $ to be $\vert X \vert $ times the common value of the matrices \eqref{eq:km}.
\end{definition}
\noindent Referring to Definition \ref{def:e0}, the matrix $e_0$ is symmetric. Moreover $e_0$ is contained in the common span of $ME^*_0 M$ and $M^*E_0 M^*$.
In the next result we describe another property of $e_0$.
\medskip

\noindent Let $Z(T)$ denote the center of $T$.

\begin{proposition} {\rm (See \cite[Corollary~11.3, Proposition~11.4]{egge1}.)} For the matrix $e_0$ from Definition \ref{def:e0}, we have
\begin{enumerate}
\item[\rm (i)] $e_0 \in Z(T)$;
\item[\rm (ii)] $E_0 e_0  = E_0$;
\item[\rm (iii)] $E^*_0 e_0  = E^*_0$;
\item[\rm (iv)] $e_0^2 = e_0$.
\end{enumerate}
\end{proposition}
\begin{proof} (i) Using
$e_0 = \vert X \vert \sum_{i=0}^D k^{-1}_i E^*_i E_0 E^*_i $ we see that $e_0$ commutes with $E^*_\ell$ for $0 \leq \ell \leq D$.
Therefore $e_0$ commutes with everything in $M^*$.
Using $e_0 = \vert X\vert \sum_{j=0}^D m^{-1}_j E_j E^*_0 E_j$ we see that $e_0$ commutes with $E_\ell$ for $0 \leq \ell \leq D$.
Therefore $e_0$ commutes with everything in $M$. The result follows since $T$ is generated by $M,M^*$.
\\
\noindent (ii) Observe
\begin{align*} 
E_0 e_0 = E_0 \vert X \vert \sum_{j=0}^D m^{-1}_j E_j E^*_0 E_j = \vert X \vert E_0 E^*_0 E_0 = E_0.
\end{align*} 
\noindent (iii) Observe
\begin{align*}
E^*_0 e_0 = E^*_0 \vert X \vert \sum_{i=0}^D k^{-1}_i E^*_i E_0 E^*_i = \vert X \vert E^*_0 E_0 E^*_0 = E^*_0.
\end{align*}
\noindent  (iv) Observe
\begin{align*}
(I-e_0)e_0 \in (I-e_0) {\rm Span}\bigl(M E^*_0 M \bigr)= {\rm Span}\bigl(M (I-e_0) E^*_0 M\bigr) = {\rm Span}\bigl(M 0 M\bigr) = 0.
\end{align*}
\end{proof}
\noindent We will say more about $e_0$ in the next section.
\medskip

\noindent We finish this section with a comment.

\begin{lemma} \label{lem:need}
For $B \in M$ we have the logical implications
\begin{align*}
B=0\quad  \Leftrightarrow\quad BE^*_0=0 \quad \Leftrightarrow \quad E^*_0B=0.
\end{align*}
Moreover for $C \in M^*$ we have the logical implications
\begin{align*}
C=0 \quad \Leftrightarrow \quad CE_0=0 \quad \Leftrightarrow\quad  E_0C=0.
\end{align*}
\end{lemma}
\begin{proof} By Lemma \ref{lem:r5} the following are nonzero for $0 \leq i \leq D$:
\begin{align*}
E_i E^*_0, \quad E^*_0 E_i, \quad E^*_i E_0, \quad E_0 E^*_i.
\end{align*}
\noindent The result is a routine consequence of this.
\end{proof}

\section{The primary $T$-module}
\noindent Throughout this section $\Gamma=(X, \mathcal R)$ denotes a distance-regular
graph with diameter $D\geq 3$. 
Fix $x \in X$ and write $T=T(x)$. Our next goal is to describe the primary $T$-module \cite[Section~5]{curtin1}, \cite[Sections~8, 9]{egge1}, \cite{nomIS},
\cite[Lemma~3.6]{terwSub1}.
Recall the vector ${\bf 1} = \sum_{y \in X} \hat y$.
For $0 \leq i \leq D$ define the vector ${\bf 1}_i = \sum_{y \in \Gamma_i(x)} \hat y$. Observe that
\begin{align*} 
A_i \hat x = {\bf 1}_i = E^*_i {\bf 1} \qquad \qquad (0 \leq i \leq D).
\end{align*}
\noindent Consequently
    \begin{align}
     M E^*_0V = M^* E_0V. 
\label{eq:onei}
\end{align}

\begin{lemma} \label{lem:prim} The vector space $M E^*_0V = M^* E_0V$ is an irreducible $T$-module.
\end{lemma}
\begin{proof} Define $\mathcal V=M E^*_0V = M^* E_0V$. We have $M \mathcal V \subseteq \mathcal V$
since $\mathcal V=M E^*_0V$. We have $M^* \mathcal V  \subseteq \mathcal V$ since $\mathcal V=M^* E_0V$. Therefore $T\mathcal V\subseteq \mathcal V$, so $\mathcal V$ is a $T$-module.
We show that the $T$-module $\mathcal V$ is irreducible. The standard $T$-module $V$ is a direct sum of irreducible $T$-modules.
There exists an irreducible $T$-module that is not orthogonal to $\hat x$. 
This $T$-module contains $\hat x$, so it contains $M \hat x =\mathcal V$. This $T$-module must equal $\mathcal V$ by irreducibility. 
\end{proof}
\begin{definition}\rm \label{def:pm} Define $\mathcal V = M E^*_0V = M^* E_0V$. The $T$-module $\mathcal V$  is called {\it primary}. 
\end{definition}

\begin{lemma} \label{lem:pm2} For $0 \leq i \leq D$ we have
\begin{align} \label{eq:pm2}
 A^*_i {\bf 1} = \vert X \vert  E_i \hat x.
 \end{align}
 \end{lemma}
 \begin{proof} Both vectors in \eqref{eq:pm2} have $y$-coordinate $\vert X \vert (E_i)_{x,y}$ for $y \in X$.
\end{proof}
\begin{definition}\rm For $0 \leq i \leq D$ let ${\bf 1}_i^*$ denote the common  vector in
\eqref{eq:pm2}.
\end{definition}

\noindent We clarify the definitions.  Note that ${\bf 1}_0 = \hat x$ and ${\bf 1}^*_0 = {\bf 1}$. Moreover
\begin{align*}
{\bf 1}^*_0 = \sum_{i=0}^D {\bf 1}_i, \qquad \qquad {\bf 1}_0 = \vert X \vert^{-1} \sum_{i=0}^D {\bf 1}^*_i.
\end{align*}

\noindent The following result is routinely verified.
\begin{lemma} \label{lem:PB} For the primary $T$-module $\mathcal V$,
\begin{enumerate}
\item[\rm (i)] ${\bf 1}_i$ is a basis for $E^*_i\mathcal V$ $(0 \leq i \leq D)$;
\item[\rm (ii)] $\lbrace {\bf 1}_i \rbrace_{i=0}^D$ is a basis for $\mathcal V$;
\item[\rm (iii)] ${\bf 1}^*_i$ is a basis for $E_i\mathcal V$  $(0 \leq i \leq D)$;
\item[\rm (iv)] $\lbrace {\bf 1}^*_i \rbrace_{i=0}^D$ is a basis for $\mathcal V$.
\end{enumerate}
\end{lemma}

\noindent Next we explain how the primary $T$-module $\mathcal V$ is related to the matrix $e_0$ from Definition  \ref{def:e0}.
Let $\mathcal V^\perp$ denote the orthogonal complement of $\mathcal V$ in $V$. We have an orthogonal direct sum $V = \mathcal V + \mathcal V^\perp$.
\begin{lemma} With the above notation,
\begin{enumerate}
\item[\rm (i)] $(e_0 - I)\mathcal V=0$;
\item[\rm (ii)] $e_0 \mathcal V^\perp = 0 $.
\end{enumerate}
\noindent In other words, $e_0$ acts on $V$ as the orthogonal projection $V \to \mathcal V$.
\end{lemma}
\begin{proof} (i) Observe
\begin{align*}
(e_0-I) \mathcal V = (e_0-I)ME^*_0V = M (e_0-I)E^*_0V = 0.
\end{align*}
\noindent (ii) Note that $\mathcal V^\perp$ is a $T$-module, so $e_0\mathcal V^\perp \subseteq \mathcal V^\perp$.
Also,
\begin{align*}
e_0 \mathcal V^\perp \subseteq e_0 V \subseteq {\rm Span}\bigl(ME^*_0 M V\bigr)  \subseteq ME^*_0V = \mathcal V.
\end{align*}
Therefore,
\begin{align*}
e_0 \mathcal V^\perp \subseteq \mathcal V^\perp \cap \mathcal V = 0.
\end{align*}
\end{proof}

\noindent We saw in Lemma \ref{lem:PB} that  $\lbrace {\bf 1}_i\rbrace_{i=0}^D$ and  $\lbrace {\bf 1}^*_i\rbrace_{i=0}^D$ are bases for the primary $T$-module $\mathcal V$.
Next we describe how these bases are related.

\begin{lemma} \label{lem:primTrans}
For $0 \leq j \leq D$ we have
\begin{enumerate}
\item[\rm (i)] ${\bf 1}_j = \vert X \vert^{-1} k_j \sum_{i=0}^D u_j(\theta_i) {\bf 1}^*_i$;
\item[\rm (ii)] ${\bf 1}^*_j = m_j \sum_{i=0}^D u_i(\theta_j) {\bf 1}_i$.
\end{enumerate}
\end{lemma}
\begin{proof} (i) Observe
\begin{align*}
{\bf 1}_j = A_j {\hat x} = k_j \sum_{i=0}^D u_j(\theta_i) E_i {\hat x} = \vert X \vert^{-1} k_j \sum_{i=0}^D u_j(\theta_i) {\bf 1}^*_i.
\end{align*}
\noindent (ii) Observe
\begin{align*}
{\bf 1}^*_j = \vert X \vert E_j {\hat x} =  m_j \sum_{i=0}^D u_i(\theta_j) A_i {\hat x} = m_j \sum_{i=0}^D u_i(\theta_j) {\bf 1}_i.
\end{align*}
\end{proof}

\noindent Next we describe how the algebra $T$ acts on the bases $\lbrace {\bf 1}_i\rbrace_{i=0}^D$ and  $\lbrace {\bf 1}^*_i\rbrace_{i=0}^D$.

\begin{lemma}\label{lem:Taction} 
For $0 \leq i,j\leq D$ we have
\begin{enumerate}
\item[\rm (i)] $E^*_i {\bf 1}_j = \delta_{i,j} {\bf 1}_j$;
\item[\rm (ii)] $A^*_i {\bf 1}_j = m_i u_j(\theta_i) {\bf 1}_j$;
\item[\rm (iii)] $E_i {\bf 1}_j = \vert X \vert^{-1} m_i k_j u_j(\theta_i) \sum_{h=0}^D u_h(\theta_i) {\bf 1}_h$;
\item[\rm (iv)] $A_i {\bf 1}_j = \sum_{h=0}^D p^h_{i,j} {\bf 1}_h$.
\end{enumerate}
\end{lemma} 
\begin{proof} These are routinely checked using the reduction rules in Lemmas \ref{lem:r3}, \ref{lem:r4} along with Lemma \ref{lem:primTrans}.
\end{proof}

\begin{lemma}\label{lem:Taction2} 
For $0 \leq i,j\leq D$ we have
\begin{enumerate}
\item[\rm (i)] $E_i {\bf 1}^*_j = \delta_{i,j} {\bf 1}^*_j$;
\item[\rm (ii)] $A_i {\bf 1}^*_j = k_i u_i(\theta_j) {\bf 1}^*_j$;
\item[\rm (iii)] $E^*_i {\bf 1}^*_j = \vert X \vert^{-1} k_i m_j u_i(\theta_j) \sum_{h=0}^D u_i(\theta_h) {\bf 1}^*_h$;
\item[\rm (iv)] $A^*_i {\bf 1}^*_j = \sum_{h=0}^D q^h_{i,j} {\bf 1}^*_h$.
\end{enumerate}
\end{lemma} 
\begin{proof} These are routinely checked using the reduction rules in Lemmas \ref{lem:r3}, \ref{lem:r4} along with Lemma \ref{lem:primTrans}.
\end{proof}

\noindent Next we bring in the bilinear form.
\begin{lemma}\label{lem:bil}
For $0 \leq i,j\leq D$ we have
\begin{enumerate}
\item[\rm (i)] $\langle {\bf 1}_i, {\bf 1}_j \rangle = \delta_{i,j} k_i$;
\item[\rm (ii)] $\langle {\bf 1}^*_i, {\bf 1}^*_j \rangle = \delta_{i,j}\vert X \vert m_i$;
\item[\rm (iii)]  $\langle {\bf 1}_i, {\bf 1}^*_j \rangle =    k_i m_j u_i(\theta_j)$.
\end{enumerate}
\end{lemma}
\begin{proof} (i) Routine. \\
\noindent (ii) Observe
\begin{align*}
\langle {\bf 1}^*_i, {\bf 1}^*_j \rangle=\vert X \vert^2 \langle E_i {\hat x}, E_j {\hat x} \rangle =   \vert X \vert^2 \langle {\hat x}, E_i E_j {\hat x} \rangle                
  = \delta_{i,j}\vert X \vert^2 \langle {\hat x}, E_i {\hat x} \rangle
 = \delta_{i,j}\vert X \vert m_i.
\end{align*}
\noindent (iii) Observe
\begin{align*}
\langle {\bf 1}_i, {\bf 1}^*_j \rangle =  
\vert X \vert \langle A_i {\hat x}, E_j {\hat x} \rangle =
\vert X \vert \langle  {\hat x}, A_i E_j {\hat x} \rangle =
\vert X \vert v_i(\theta_j) \langle  {\hat x},  E_j {\hat x} \rangle =
  k_i m_j u_i(\theta_j).
\end{align*}
\end{proof}

\section{The Krein condition and the triple product relations}
The Krein condition states that the Krein parameters are nonnegative. In this section we will prove the Krein condition, and 
 also derive some relations called the triple product relations.
\medskip

\noindent Throughout this section $\Gamma=(X,\mathcal R)$ denotes a distance-regular graph with diameter $D\geq 3$. Fix $x \in X$ and write $T=T(x)$.
\medskip

\noindent In the following result, the second item is a variation on \cite[Proposition~5.1]{norton}.
\begin{lemma} \label{lem:eaeB}  {\rm (See \cite[Proposition~5.1]{norton}, \cite[Lemmas~3.1, 4.1]{dickie2}.)}
For $0 \leq h,i,j,r,s,t\leq D$,
\begin{enumerate}
\item[\rm (i)] $\langle E^*_h A_i E^*_j, E^*_r A_s E^*_t\rangle = \delta_{h,r} \delta_{i,s} \delta_{j,t} k_h p^h_{i,j}$;
\item[\rm (ii)] $\langle E_h A^*_i E_j, E_r A^*_s E_t\rangle = \delta_{h,r} \delta_{i,s} \delta_{j,t} m_h q^h_{i,j}$.
\end{enumerate}
\end{lemma}
\begin{proof} (i) Using ${\rm tr}(BC) = {\rm tr}(CB)$,
\begin{align*}
\langle E^*_h A_i E^*_j, E^*_r A_s E^*_t\rangle 
&= {\rm tr} \bigl( E^*_h A_i E^*_j (E^*_r A_s E^*_t)^t \bigr) \\
&= {\rm tr} \bigl( E^*_h A_i E^*_j E^*_t A_s E^*_r \bigr) \\
&= \delta_{h,r}\delta_{j,t} {\rm tr} \bigl( E^*_h A_i E^*_j A_s \bigr) 
\end{align*}
and
\begin{align*}
 {\rm tr} \bigl( E^*_h A_i E^*_j A_s \bigr) &= \sum_{y \in X} \sum_{z \in X} (E^*_h)_{y,y} (A_i)_{y,z} (E^*_j)_{z,z} (A_s)_{z,y} \\
 &= \sum_{y \in X} \sum_{z \in X} (E^*_h)_{y,y} (A_i \circ A_s)_{y,z} (E^*_j)_{z,z}  \\
 &= \delta_{i,s} \sum_{y \in X} \sum_{z \in X} (E^*_h)_{y,y} (A_i)_{y,z} (E^*_j)_{z,z}  \\
&=\delta_{i,s} \sum_{\stackrel{ \stackrel{\scriptstyle y \in \Gamma_h(x), }{\scriptstyle z \in \Gamma_j(x),}}{ \scriptstyle \partial(y,z)=i}} 1\\
&= \delta_{i,s} k_h p^h_{i,j}.
\end{align*}
\noindent (ii) We have
\begin{align*}
\langle E_h A^*_i E_j, E_r A^*_s E_t\rangle 
&= {\rm tr} \bigl( E_h A^*_i E_j (E_r A^*_s E_t)^t \bigr) \\
&= {\rm tr} \bigl( E_h A^*_i E_j E_t A^*_s E_r \bigr) \\
&= \delta_{h,r}\delta_{j,t} {\rm tr} \bigl( E_h A^*_i E_j A^*_s \bigr) 
\end{align*}
and
\begin{align*}
{\rm tr} \bigl( E_h A^*_i E_j A^*_s \bigr)
&= 
 \sum_{y \in X} \sum_{z \in X} (E_h)_{y,z} (A^*_i)_{z,z} (E_j)_{z,y} (A^*_s)_{y,y} \\
&=\vert X \vert^2  \sum_{y \in X} \sum_{z \in X} (E_h)_{y,z} (E_i)_{x,z} (E_j)_{z,y} (E_s)_{x,y} \\ 
 &= \vert X \vert^2 \sum_{y \in X} \sum_{z \in X} (E_s)_{x,y} (E_h \circ E_j)_{y,z} (E_i)_{z,x}  \\
 &= \vert X \vert^2 \Bigl( {\mbox{{\rm $(x,x)$-entry of }}} E_s(E_h \circ E_j) E_i\Bigr) \\
& = \vert X \vert {\rm tr} \bigl( E_s (E_h \circ E_j) E_i \bigr) \\
& = \vert X \vert {\rm tr} \bigl((E_h \circ E_j) E_i E_s\bigr) \\
& = \delta_{i,s} \vert X \vert {\rm tr} \bigl((E_h \circ E_j) E_i \bigr) \\
& = \delta_{i,s} \sum_{\ell=0}^D q^\ell_{h,j} {\rm tr} (E_{\ell} E_i) \\
&= \delta_{i,s} q^i_{h,j} {\rm tr} (E_i) \\
&= \delta_{i,s} q^i_{h,j} m_i \\
& = \delta_{i,s} m_h q^h_{i,j}.
\end{align*}
\end{proof}

\begin{corollary} \label{cor:eae} For $0 \leq h,i,j\leq D$ we have
\begin{enumerate}
\item[{\rm (i)}]$\Vert E^*_h A_i E^*_j \Vert^2 = k_h p^h_{i,j}$;
\item[{\rm (ii)}] $\Vert E_h A^*_i E_j \Vert^2 = m_h q^h_{i,j}$.
\end{enumerate}
\end{corollary} 
\begin{proof} Set $r=h$, $s=i$, $t=j$ in Lemma \ref{lem:eaeB}.
\end{proof}

\noindent The following result is called the {\it Krein condition}. See \cite[p.~69]{bannai} for a discussion of the history.
\begin{theorem} \label{thm:krein}   
We have $q^h_{i,j}\geq 0$ for $0 \leq h,i,j\leq D$.
\end{theorem}
\begin{proof} By Corollary \ref{cor:eae}(ii) and since $\Vert B \Vert^2 \geq 0$ for all $B \in {\rm Mat}_X(\mathbb R)$.
\end{proof}

\begin{theorem} \label{lem:tpr} {\rm (See \cite[Lemma~3.2]{terwSub1}.)}
For $0 \leq h,i,j\leq D$ we have
\begin{enumerate}
\item[\rm (i)] $E^*_h A_i E^*_j = 0$ if and only if  $p^h_{i,j}=0$;
\item[\rm (ii)] $E_h A^*_i E_j = 0 $ if and only if  $q^h_{i,j}=0$.
\end{enumerate}
\end{theorem}
\begin{proof} By  Corollary \ref{cor:eae} and since $\Vert B \Vert^2 = 0$ implies $B=0$ for all $B \in {\rm Mat}_X(\mathbb R)$.
\end{proof}
\noindent The relations in Theorem \ref{lem:tpr} are called the {\it triple product relations}.
\medskip

\noindent We bring in some notation. For subspaces $R, S$ of ${\rm Mat}_X(\mathbb R)$, define
\begin{align*}
RS = {\rm Span} \lbrace rs \vert r \in R, \; s \in S\rbrace.
\end{align*}
\begin{theorem} {\rm (See \cite[Section~7]{sum2}.)} With the above notation,
\begin{enumerate}
\item[\rm (i)] the vector space $M^*MM^*$ has an orthogonal basis
\begin{align*} 
\lbrace E^*_hA_i E^*_j \vert 0 \leq h,i,j\leq D, \; p^h_{i,j}\not=0\rbrace;
\end{align*}
\item[\rm (ii)]  the vector space $MM^*M$ has an orthogonal basis
\begin{align*} 
\lbrace E_h A^*_i E_j \vert 0 \leq h,i,j\leq D, \; q^h_{i,j}\not=0\rbrace.
\end{align*}
\end{enumerate}
\end{theorem}
\begin{proof} By Lemma \ref{lem:eaeB}
and Theorem \ref{lem:tpr}.
\end{proof}
\noindent We mention a consequence of Theorem \ref{lem:tpr}.
\begin{proposition}
\label{prop:tpr} For $ 0 \leq i,j\leq D$ we have
\begin{align} \label{eq:tpr}
A_i E^*_j V \subseteq    \sum_{\stackrel{ \scriptstyle 0 \leq h \leq D }{ \scriptstyle p^h_{i,j} \not=0}} E^*_hV,
\qquad \qquad 
A^*_i E_j V \subseteq    \sum_{\stackrel{ \scriptstyle 0 \leq h \leq D }{ \scriptstyle q^h_{i,j} \not=0}} E_hV.
\end{align}
\end{proposition}
\begin{proof} Concerning the containment on the left in \eqref{eq:tpr},
\begin{align*}
A_i E^*_j V = I A_i E^*_j V = \sum_{h=0}^D E^*_h A_i E^*_j V =  \sum_{\stackrel{ \scriptstyle 0 \leq h \leq D }{ \scriptstyle p^h_{i,j} \not=0}} E^*_hA_i E^*_jV 
\subseteq  \sum_{\stackrel{ \scriptstyle 0 \leq h \leq D }{ \scriptstyle p^h_{i,j} \not=0}} E^*_hV.
\end{align*}
The containment on the right in \eqref{eq:tpr} is similarly obtained.
\end{proof}

\noindent We finish this section with some comments about the primary $T$-module.

\begin{lemma} \label{lem:tprP} For the primary $T$-module $\mathcal V$ and $0 \leq h,i,j\leq D$,
\begin{enumerate}
\item[\rm (i)] $E^*_hA_i E^*_j=0$ on $\mathcal V$ if and only if $p^h_{i,j}=0$;
\item[\rm (ii)] $E_hA^*_i E_j=0$ on $\mathcal V$ if and only if $q^h_{i,j}=0$.
\end{enumerate}
\end{lemma}
\begin{proof} Use parts (i), (iv) of Lemmas  \ref{lem:Taction},  \ref{lem:Taction2}.
\end{proof}

\begin{lemma} \label{lem:eae} For $0 \leq i,j \leq D$ the following holds on the primary $T$-module $\mathcal V$:
\begin{align*}
E^*_i A_j E^*_i = p^i_{i,j} E^*_i, \qquad \qquad E_i A^*_j E_i = q^i_{i,j} E_i.
\end{align*}
\end{lemma}
\begin{proof} Use parts (i), (iv) of Lemmas  \ref{lem:Taction},  \ref{lem:Taction2}.
\end{proof}

\section{The function algebra and the Norton algebra}

Throughout this section $\Gamma=(X,\mathcal R)$ denotes a distance-regular graph with diameter $D\geq 3$. 
In  Theorem \ref{lem:tpr} we saw that each vanishing Krein parameter gives a triple product relation. In this section
we consider the vanishing Krein parameters from another point of view, due to Cameron, Goethals, and Seidel \cite[Proposition~5.1]{norton}.
We will also briefly mention Norton algebras.
\medskip

\noindent Recall the basis $\lbrace {\hat y} \rbrace_{y \in X}$ for the standard module $V$.

\begin{definition}\label{def:circ} \
\rm We turn the vector space $V$  into a commutative, associative, $\mathbb R$-algebra with multiplication $\circ$ defined as follows:
\begin{align} \label{eq:circ}
{\hat y} \circ {\hat z} = \delta_{y,z} {\hat y} \qquad \qquad y, z \in X.
\end{align}
The algebra $V$ is isomorphic to the algebra of functions $X \to \mathbb R$. Motivated by this, we call the algebra $V$ the {\it function algebra}.
\end{definition}
\noindent In order to illustrate the multiplication $\circ$, let $v, w \in V$ and write
\begin{align*}
v = \sum_{y \in X} v_y {\hat y}, \qquad \qquad w = \sum_{y \in X} w_y {\hat y}     \qquad \qquad v_y, w_y \in \mathbb R.
\end{align*}
\noindent Then
\begin{align*}
v \circ w = \sum_{y \in X} v_y w_y {\hat y}.
\end{align*}

\begin{lemma} \label{lem:one} For the function algebra $V$, the multiplicative identity is ${\bf 1} = \sum_{y \in X}{\hat y}$.
\end{lemma} 
\begin{proof} Routine.
\end{proof}

\noindent For the rest of this section, fix $x \in X$ and write $T=T(x)$.
\begin{lemma} \label{lem:cv}
For $v \in V $ and $0 \leq i \leq D$,
\begin{align}
\label{eq:cv}
A^*_i v = \vert X \vert E_i {\hat x} \circ v.
\end{align}
\end{lemma}
\begin{proof} Write $v = \sum_{y \in X} v_y {\hat y}$. Pick $y \in X$. The $y$-coordinate of $A^*_iv$ is
\begin{align*}
(A^*_iv)_y = (A^*_i)_{y,y} v_y = \vert X \vert (E_i)_{x,y} v_y.
\end{align*}
The $y$-coordinate of $E_i {\hat x} \circ v$ is
\begin{align*}
\bigl( E_i {\hat x} \circ v \bigr)_y = (E_i {\hat x})_y v_y = (E_i)_{y,x} v_y  = (E_i)_{x,y} v_y.
\end{align*}
The result follows.
\end{proof}

\noindent We bring in some notation. For subspaces $R, S $ of $V$ define
\begin{align} \label{eq:RcS}
R \circ S = {\rm Span} \lbrace r \circ s \vert r \in R,\; s \in S\rbrace.
\end{align}

\begin{theorem}\label{thm:fa} {\rm (See \cite[Proposition~5.1]{norton}.)} For $0 \leq i,j\leq D$,
\begin{align} \label{lem:fa}
E_iV \circ E_j V = \sum_{\stackrel{ \scriptstyle 0 \leq h \leq D }{ \scriptstyle q^h_{i,j} \not=0}} E_hV.
\end{align}
\end{theorem}
\begin{proof} We first establish the inclusion $\subseteq$. By construction $E_i V = {\rm Span} \lbrace E_i {\hat y} \vert y \in X\rbrace$.
We show that for $ y \in X$,
\begin{align*} 
E_i {\hat y}  \circ E_j V \subseteq \sum_{\stackrel{ \scriptstyle 0 \leq h \leq D }{ \scriptstyle q^h_{i,j} \not=0}} E_hV.
\end{align*}
Since our base vertex $x$ is arbitrary, we may assume without loss of generality that $x=y$. By Proposition \ref{prop:tpr} and Lemma \ref{lem:cv},
\begin{align*} 
E_i {\hat x}  \circ E_j V = A^*_i E_j V 
\subseteq \sum_{\stackrel{ \scriptstyle 0 \leq h \leq D }{ \scriptstyle q^h_{i,j} \not=0}} E_hV.
\end{align*}
\\
\noindent Next we establish the inclusion $\supseteq$. For $0 \leq h \leq D$ such that $q^h_{i,j} \not=0$, we show that
$E_i V \circ E_j V \supseteq E_hV$. We have
\begin{align*}
E_i V \circ E_j V &= {\rm Span} \lbrace E_i {\hat y} \circ E_j {\hat z} \vert y,z \in X \rbrace \\
&\supseteq {\rm Span} \lbrace E_i {\hat y} \circ E_j {\hat y}\vert y  \in X \rbrace \\
&= {\rm Span} \lbrace (E_i  \circ E_j ){\hat y}\vert y \in X \rbrace \\
&= (E_i \circ E_j) V\\
&\supseteq (E_i \circ E_j) E_h V\\
&= \Biggl( \vert X \vert^{-1} \sum_{\ell=0}^D q^\ell_{i,j} E_\ell \Biggr) E_hV\\
&= \vert X \vert^{-1} q^h_{i,j} E_hV \\
&= E_hV.
\end{align*}
\end{proof}


\noindent Next we briefly review the Norton algebra.

\begin{lemma} \label{lem:norton} {\rm (See \cite[Proposition~5.2]{norton}.)}
For $0 \leq j \leq D$ we endow $E_j V$ with a binary operation $\star$ as follows:
\begin{align*} 
u \star v = E_j(u \circ v) \qquad \qquad u, v \in E_jV.
\end{align*}
Then for $u,v,w \in E_jV$ and $\alpha \in \mathbb R$,
\begin{enumerate}
\item[\rm (i)] $u \star v = v \star u$;
\item[\rm (ii)] $u \star (v + w) = u \star v + u \star w$;
\item[\rm (iii)] $(\alpha u)\star v = \alpha (u \star v) $.
\end{enumerate}
\end{lemma}
\begin{proof} This is routinely checked.
\end{proof}

\noindent  Referring to Lemma \ref{lem:norton}, the vector space $E_jV$ together with the opertion $\star$ is
called the $j^{\rm th}$ {\it Norton algebra} for $\Gamma$; see \cite{norton}. This algebra is commutative and nonassociative. It has no multiplicative 
identity in general. See \cite{jia, levNorton, Norton, nortonPT} for recent results on the Norton algebra.

\section{The function algebra and nondegenerate primitive idempotents}

Throughout this section $\Gamma=(X,\mathcal R)$ denotes a distance-regular graph with diameter $D\geq 3$. 
Recall the function
algebra $V$ from Definition \ref{def:circ}. We will show that a primitive idempotent $E$ of $\Gamma$ is nondegenerate if and only if
the eigenspace $EV$ generates  $V$ in the function algebra.
\medskip

\noindent 
For a subspace $U \subseteq V$ we describe the subalgebra of  $V$ generated by $U$. This subalgebra contains $\bf 1$ by Lemma \ref{lem:one}.
To see what else is in the subalgebra, we define a binary relation on $X$ called $U$-equivalence.
\begin{definition} \label{def:WE}
\rm Vertices $y, z$ in $ X$ are said to be {\it $U$-equivalent} whenever 
for all $u \in U$, the $y$-coordinate of $u$ is equal to the $z$-coordinate of $u$.
Observe that $U$-equivalence  is an equivalence relation.
\end{definition}

\begin{definition}\label{def:hat} \rm
For a subset $Y \subseteq X$ define $ \hat Y = \sum_{y \in Y} {\hat y}$.
\end{definition}

\begin{lemma} \label{lem:W} For a subspace $U \subseteq V$ the following are equal:
\begin{enumerate}
\item[\rm (i)] the subalgebra of the function algebra $V$ generated by $U$;
\item[\rm (ii)] ${\rm Span} \lbrace {\hat Y} | \mbox{ $Y$ is a $U$-equivalence class} \rbrace$.
\end{enumerate}
\end{lemma}
\begin{proof} ${\rm (i)} \subseteq {\rm (ii)}$: Note that ${\rm Span} \lbrace {\hat Y} | \mbox{ $Y$ is a $U$-equivalence class} \rbrace$
is a subalgebra of $V$ that contains $U$.
\\
\noindent ${\rm (i)} \supseteq {\rm (ii)}$: Let $Y$ denote a $U$-equivalence class. We show that $\hat Y$ is contained in the subalgebra of $V$
generated by $U$. List the $U$-equivalence classes $Y=Y_0, Y_1, \ldots, Y_n$. For $u\in U$ write
\begin{align*} 
u = \sum_{i=0}^n \alpha_i (u) {\hat Y}_i \qquad \qquad \alpha_i(u) \in \mathbb R.
\end{align*}
For $1 \leq i \leq n$ there exists $u_i \in U$ such that $\alpha_0 (u_i) \not=\alpha_i(u_i)$.
We have
\begin{align*}
{\hat Y} = \prod_{i=1}^n \frac{ u_i - \alpha_i(u_i) {\bf 1} }{ \alpha_0(u_i) - \alpha_i(u_i)},
\end{align*}
where the product is with respect to $\circ$. Therefore $\hat Y$ is contained in the subalgebra of $V$ generated by $U$.
\end{proof}

\begin{corollary}
\label{cor:genV}
 For a subspace $U \subseteq V$ the following are equivalent:
\begin{enumerate}
\item[\rm (i)] $U$ generates the function algebra $V$;
\item[\rm (ii)] each $U$-equivalence class has cardinality one.
\end{enumerate}
\end{corollary}
\begin{proof} By Lemma \ref{lem:W}.
\end{proof}

\noindent We have been discussing a subspace $U$ of $V$. Next we consider the special case in which
$U=EV$, where $E$ is a primitive idempotent of $\Gamma$. 
Let us compute the $EV$-equivalence classes. We have
$EV = {\rm Span} \lbrace E{\hat w} \vert w \in X\rbrace$. Above Lemma \ref{lem:uij} we saw that for $y, z \in X$,
\begin{align}
{\mbox{ \rm $y$-coordinate of $E{\hat z}$}}
={\mbox{ \rm $z$-coordinate of $E{\hat y}$}}.
\label{eq:chain}
\end{align}

\begin{lemma} \label{lem:primP} Let $E$ denote a primitive idempotent of $\Gamma$. Then for $y, z \in X$
the following are equivalent:
\begin{enumerate}
\item[\rm (i)] $y, z$ are in the same $EV$-equivalence class;
\item[\rm (ii)]  $E{\hat y} = E {\hat z}$.
\end{enumerate}
\end{lemma} 
\begin{proof} Condition (i) holds, if and only  if the $y$-coordinate of $E{\hat w}$ is equal to the $z$-coordinate of $E{\hat w}$ for all $w \in X$.
Condition (ii) holds, if and only if the $w$-coordinate of $E{\hat y}$ is equal to the $w$-coordinate of $E{\hat z}$ for all $w \in X$.
By these comments and \eqref{eq:chain}, the conditions (i), (ii) are equivalent.
\end{proof}

\begin{theorem} \label{thm:Egen} Let $E$ denote a primitive idempotent of $\Gamma$.
Then the following are equivalent:
\begin{enumerate}
\item[\rm (i)] $EV$ generates the function algebra $V$;
\item[\rm (ii)] $E$ is nondegenerate.
\end{enumerate}
\end{theorem}
\begin{proof} By Corollary \ref{cor:genV} and
Lemma \ref{lem:primP} we 
have the logical implications
\begin{align*}
&{\mbox{\rm $EV$ generates the function algebra $V$} }
\\
& \Leftrightarrow {\mbox{\rm each $EV$-equivalence class has cardinality one}}
\\
& \Leftrightarrow {\mbox {$ \lbrace E {\hat y}\rbrace_{y \in X}$ are mutually distinct}}
\\
& \Leftrightarrow {\mbox{ \rm $E$ is nondegenerate}}.
\end{align*}
\end{proof}

\section{The $Q$-polynomial property and Askey-Wilson duality}
\noindent In this section we discuss the $Q$-polynomial property and its connection to Askey-Wilson duality.
\medskip

\noindent
Throughout this section $\Gamma=(X,\mathcal R)$ denotes a distance-regular graph with diameter $D\geq 3$. Recall the primitive idempotents
$\lbrace E_i \rbrace_{i=0}^D$ of $\Gamma$.

\begin{definition} \label{def:qpoly} \rm The ordering $\lbrace E_i \rbrace_{i=0}^D$ is called {\it $Q$-polynomial} whenever the following hold
for $0 \leq h,i,j\leq D$:
\begin{enumerate}
\item[\rm (i)] $q^h_{i,j}=0$ if one of $h,i,j$ is greater than the sum of the other two;
\item[\rm (ii)]  $q^h_{i,j}\not=0$ if one of $h,i,j$ is equal to the sum of the other two.
\end{enumerate}
\end{definition}

\begin{definition}\rm We say that $\Gamma$ is {\it $Q$-polynomial} whenever there exists at least one $Q$-polynomial ordering
of the primitive idempotents.
\end{definition}

\noindent For the rest of this section, we  assume that the ordering $\lbrace E_i \rbrace_{i=0}^D$ is $Q$-polynomial. Define
\begin{align*}
c^*_i = q^i_{1,i-1} \;\; (1 \leq i \leq D), \qquad a^*_i = q^i_{1,i} \;\; (0 \leq i \leq D), \qquad  b^*_i = q^i_{1,i+1} \;\; (0 \leq i \leq D-1).
\end{align*}
\noindent  Note that $a^*_0=0$ and $c^*_1=1$. Moreover
\begin{align*}
c^*_i > 0 \quad (1 \leq i \leq D), \qquad \qquad b^*_i > 0 \quad (0 \leq i \leq D-1).
\end{align*}  From Lemma \ref{lem:details}(iv) we obtain
\begin{align*}
c^*_i + a^*_i + b^*_i = m_1 \qquad \qquad (0 \leq i \leq D),
\end{align*}
where $c^*_0=0$ and $b^*_D=0$. By Lemma \ref{lem:mq} we have $m_i c^*_i = m_{i-1} b^*_{i-1} $ for $1 \leq i \leq D$. Consequently
\begin{align} \label{eq:miform}
m_i = \frac{ b^*_0 b^*_1 \cdots b^*_{i-1}}{c^*_1 c^*_2 \cdots c^*_i} \qquad \qquad (0 \leq i \leq D).
\end{align}
\noindent For the rest of this section, fix $x \in X$ and write $T=T(x)$. Recall the bases $\lbrace A^*_i \rbrace_{i=0}^D$ and $\lbrace E^*_i\rbrace_{i=0}^D$
for $M^*$. We abbreviate $A^*=A^*_1$ and call this the {\it dual adjacency matrix} (with respect to $x$ and the given $Q$-polynomial structure).
\medskip

\noindent
Our next goal is to show that $A^*_i$ is a polynomial of degree $i$  in $A^*$ for $0 \leq i \leq D$.
\begin{lemma} \label{lem:As3t} We have
\begin{align*}
A^* A^*_i &= b^*_{i-1} A^*_{i-1} + a^*_i A^*_i + c^*_{i+1} A^*_{i+1} \qquad \qquad (1 \leq i \leq D-1),\\
A^* A^*_D &= b^*_{D-1} A^*_{D-1} + a^*_D A^*_D.
\end{align*}
\end{lemma} \begin{proof}
This is $A^*_i A^*_j = \sum_{h=0}^D q^h_{i,j} A^*_h$ with $j=1$.
\end{proof}

\begin{definition}\label{def:vis}\rm We define some polynomials $\lbrace v^*_i \rbrace_{i=0}^{D+1} $ in $\mathbb R\lbrack \lambda \rbrack$ such that
\begin{align*}
&v^*_0=1, \qquad \quad v^*_1 = \lambda, \\
\lambda v^*_i &= b^*_{i-1} v^*_{i-1} + a^*_i v^*_i + c^*_{i+1} v^*_{i+1} \qquad \qquad (1 \leq i \leq D),
\end{align*}
where $c^*_{D+1}=1$.
\end{definition}

\begin{lemma} \label{lem:vis} The following {\rm (i)--(iv)} hold:
\begin{enumerate}
\item[\rm (i)] ${\rm deg} \,v^*_i = i \quad (0 \leq i \leq D+1)$;
\item[\rm (ii)] the coefficient of $\lambda^i$ in $v^*_i$ is $(c^*_1 c^*_2 \cdots c^*_i)^{-1} \quad (0 \leq i \leq D+1)$;
\item[\rm (iii)] $v^*_i(A^*)=A^*_i \quad (0 \leq i \leq D)$;
\item[\rm (iv)] $v^*_{D+1}(A^*)=0$.
\end{enumerate}
\end{lemma}
\begin{proof} (i), (ii) By Definition \ref{def:vis}. \\
\noindent (iii), (iv) Compare Lemma \ref{lem:As3t} and Definition \ref{def:vis}.
\end{proof}

\begin{corollary} \label{lem:AsGen}  The following hold:
\begin{enumerate}
\item[\rm (i)]  the algebra $M^*$ is generated by $A^*$;
\item[\rm (ii)]  the minimal polynomial of $A^*$ is
$c^*_1 c^*_2 \cdots c^*_D v^*_{D+1}$.
\end{enumerate}
\end{corollary}
\begin{proof} By Lemma \ref{lem:vis} and since $\lbrace A^*_i\rbrace_{i=0}^D$ is a basis for $M^*$.
\end{proof}

\noindent Next we consider the eigenvalues of $A^*$. By Lemma \ref{lem:transAsEs}(i),
\begin{align*}
A^* = m_1 \sum_{i=0}^D u_i(\theta_1) E^*_i.
\end{align*}
Abbreviate
\begin{align}
\theta^*_i = m_1 u_i(\theta_1) \qquad \qquad (0 \leq i \leq D),
 \label{eq: dualE}
\end{align}
\noindent so that
\begin{align} \label{eq:ths}
 A^*= \sum_{i=0}^D \theta^*_i E^*_i.
 \end{align}

\begin{lemma} \label{lem:thsDist} The following {\rm (i)--(iii)} hold:
\begin{enumerate}
\item[\rm (i)] the polynomial $v^*_{D+1}$ has $D+1$ mutually distinct roots $\lbrace \theta^*_i \rbrace_{i=0}^D$;
\item[\rm (ii)] the eigenspaces of $A^*$ are $\lbrace E^*_iV \rbrace_{i=0}^D$;
\item[\rm (iii)] for $0 \leq i \leq D$, $\theta^*_i$ is the eigenvalue of $A^*$ for $E^*_iV$.
\end{enumerate}
\end{lemma}
\begin{proof} (i) The roots of $v^*_{D+1}$ are mutually distinct by Corollary \ref{lem:AsGen}(ii) and since $A^*$ is diagonal.
These roots are  $\lbrace \theta^*_i \rbrace_{i=0}^D$ by \eqref{eq:ths} .
\\
\noindent (ii), (iii) By  \eqref{eq:ths}.
\end{proof}

\begin{definition}\rm 
\label{def:dualEig} For $0 \leq i \leq D$ we call $\theta^*_i$ the {\it $i^{th}$ dual eigenvalue of $\Gamma$} (with respect to the given $Q$-polynomial
structure). 
\end{definition}

\noindent For convenience we adjust the normalization of the polynomials $v^*_i$.
\begin{definition}\rm \label{def:uis}
Define the polynomial
\begin{align} \label{eq:uis}
u^*_i = \frac{v^*_i}{m_i} \qquad \qquad (0 \leq i \leq D).
\end{align}
\end{definition}

\begin{lemma} We have
\begin{align*} 
& u^*_0 = 1, \qquad \quad u^*_1= m^{-1}_1 \lambda, \\
 \lambda u^*_i &= c^*_i u^*_{i-1} + a^*_i u^*_i + b^*_i u^*_{i+1} \qquad \qquad (1 \leq i \leq D-1),\\
 & \lambda u^*_D -c^*_D u^*_{D-1} - a^*_D u^*_D = m^{-1}_D v^*_{D+1}.
\end{align*}
\end{lemma}
\begin{proof} Evaluate the recurrence in Definition \ref{def:vis}
using $v^*_i = m_i u^*_i$ $(0 \leq i \leq D)$ and \eqref{eq:miform}.
\end{proof}

\noindent We just defined the polynomials  $\lbrace u^*_i\rbrace_{i=0}^D$. Next we explain how the polynomials $\lbrace u_i\rbrace_{i=0}^D$ and $\lbrace u^*_i\rbrace_{i=0}^D$
are related.

\begin{theorem}\label{thm:AWD} {\rm (See \cite[p.~14]{delsarte}.)} We have
\begin{align} \label{eq:AWD}
u_i(\theta_j) = u^*_j(\theta^*_i) \qquad \qquad (0 \leq i,j\leq D).
\end{align}
\end{theorem}
\begin{proof} Using Lemma  \ref{lem:AisEis} and Lemma \ref{lem:vis}(iii),
\begin{align*}
u_i(\theta_j)E^*_i = m^{-1}_j A^*_j E^*_i = m^{-1}_j v^*_j(A^*) E^*_i = u^*_j(A^*) E^*_i = u^*_j(\theta^*_i) E^*_i.
\end{align*}
The result follows.
\end{proof}

\noindent We will comment on Theorem \ref{thm:AWD} shortly.

\begin{lemma} For $0 \leq i \leq D$ we have
\begin{enumerate}
\item[\rm (i)] $v^*_i(\theta^*_0)=m_i$;
\item[\rm (ii)] $u^*_i(\theta^*_0)= 1$.
\end{enumerate}
\end{lemma}
\begin{proof} By Theorem \ref{thm:AWD} we obtain $u^*_i(\theta^*_0)= u_0 (\theta_i) = 1$, giving (ii). By Definition  \ref{def:uis}
we get (i).
\end{proof}

\noindent In Propositions \ref{prop:uorth} and \ref{prop:vorth}
we obtained some orthogonality relations for the polynomials $\lbrace u_i\rbrace_{i=0}^D$ and $\lbrace v_i \rbrace_{i=0}^D$.
Next we give the analogous orthogonality relations for the polynomials $\lbrace u^*_i\rbrace_{i=0}^D$ and $\lbrace v^*_i \rbrace_{i=0}^D$.
\begin{proposition} \label{prop:usdual} We have 
\begin{align*}
\sum_{\ell=0}^D u^*_i(\theta^*_\ell) u^*_j(\theta^*_\ell) k_\ell &= \delta_{i,j} m^{-1}_i \vert X \vert \qquad \qquad (0 \leq i,j\leq D),
\\
\sum_{\ell=0}^D u^*_\ell(\theta^*_i) u^*_\ell(\theta^*_j) m_\ell &= \delta_{i,j} k^{-1}_i \vert X \vert \qquad \qquad (0 \leq i,j\leq D).
\end{align*}
\end{proposition}
\begin{proof} 
Combine Proposition \ref{prop:uorth} and Theorem \ref{thm:AWD}.
\end{proof}

\begin{proposition} We have 
\begin{align*}
\sum_{\ell=0}^D v^*_i(\theta^*_\ell) v^*_j(\theta^*_\ell) k_\ell &= \delta_{i,j} m_i \vert X \vert \qquad \qquad (0 \leq i,j\leq D),
\\
\sum_{\ell=0}^D v^*_\ell(\theta^*_i) v^*_\ell(\theta^*_j) m^{-1}_\ell &= \delta_{i,j} k^{-1}_i \vert X \vert \qquad \qquad (0 \leq i,j\leq D).
\end{align*}
\end{proposition}
\begin{proof} Combine Definition  \ref{def:uis} and Proposition \ref{prop:usdual}.
\end{proof}

\noindent Equation \eqref{eq:AWD}   is called {\it Askey-Wilson duality} \cite{qrac}
or
{\it Delsarte duality} \cite{leonard}. 
In \cite{leonard} D. Leonard classifies the orthogonal polynomial sequences that satisfy Askey-Wilson duality.
See  \cite[p.~260]{bannai} for a more comprehensive treatment. The classification shows that the orthogonal 
polynomial sequences that satisfy Askey-Wilson duality belong to the terminating branch of the Askey scheme;
this branch consists of the $q$-Racah polynomials \cite{aw} along with their limiting cases \cite{koekoek}. 
The theory of Leonard pairs \cite{hanson2,LS99,qrac,TLT:array,aa,notesLS} provides a modern approach to Askey-Wilson duality.

\section{The function algebra characterization of the $Q$-polynomial property}

\noindent In this section we characterize the $Q$-polynomial property in terms of the function algebra.
\medskip

\noindent 
Throughout this section $\Gamma=(X,\mathcal R)$ denotes a distance-regular graph with diameter $D\geq 3$. Recall the function algebra
$V$ and the primitive idempotents
$\lbrace E_i \rbrace_{i=0}^D$.  Recall that
\begin{align*}
E_0V = \mathbb R {\bf 1}, \qquad \qquad {\bf 1} = \sum_{y \in X} {\hat y}.
\end{align*}
\noindent We remind the reader that for subspaces $R, S$ of $V$,
 \begin{align*}
 R \circ S = {\rm Span}\lbrace r \circ s | r \in R, \; s \in S\rbrace.
 \end{align*}
 \noindent For an integer $n\geq 0$ define
 \begin{align}\label{eq:copies}
 R^{\circ n} = R \circ R \circ \cdots \circ R \qquad \quad ({\mbox{\rm $n$ copies}}).
 \end{align}
 We interpret $R^{\circ 0} = \mathbb R {\bf 1}$.
\medskip

\noindent The following result appears in \cite[Lecture~23]{suz}. It is also mentioned in \cite[p.~30]{dkt}.

\begin{theorem}\label{thm:faq}
The following are equivalent:
\begin{enumerate}
\item[\rm (i)] the ordering $\lbrace E_i \rbrace_{i=0}^D$ is $Q$-polynomial;
\item[\rm (ii)] $E_1$ is nondegenerate and 
\begin{align} \label{eq:1im}
E_1 V \circ E_i V \subseteq E_{i-1} V + E_iV+ E_{i+1}V \qquad \qquad (0 \leq i \leq D),
\end{align}
where $E_{-1}=0$ and $E_{D+1}=0$;
\item[\rm (iii)] for $0 \leq i \leq D$,
\begin{align} \label{eq:expand}
\sum_{\ell = 0}^i E_\ell V = \sum_{\ell=0}^i (E_1V)^{\circ \ell}.
\end{align}
\end{enumerate}
\end{theorem}
\begin{proof} ${\rm (i)} \Rightarrow {\rm (ii)}$ For the given $Q$-polynomial structure, the dual eigenvalues are  $\theta^*_i = m_1 u_i(\theta_1)$ $(0 \leq i \leq D)$.
The scalars $\lbrace \theta^*_i \rbrace_{i=0}^D$ are mutually distinct, so $\theta^*_i \not=\theta^*_0$ for $1 \leq i \leq D$.
Therefore $E_1$ is nondegenerate. By Theorem \ref{thm:fa} we obtain
\begin{align*}
E_1V \circ E_i V = \sum_{\stackrel{ \scriptstyle 0 \leq h \leq D }{ \scriptstyle q^h_{1,i} \not=0}} E_hV \qquad \qquad (0 \leq i \leq D).
\end{align*}
The ordering $\lbrace E_i \rbrace_{i=0}^D$ is $Q$-polynomial, so 
$q^h_{1,i} = 0$ if $\vert h-i\vert >1$ $(0 \leq h,i\leq D)$. By these comments we obtain \eqref{eq:1im}.
\\
\noindent ${\rm (ii)} \Rightarrow {\rm (iii)}$ For $0 \leq i \leq D$ define  $P_i = \sum_{\ell=0}^i (E_1V)^{\circ \ell}$.
By Theorem \ref{thm:fa}, there exists a subset $S_i \subseteq \lbrace 0,1,\ldots, D\rbrace$ such that $P_i = \sum_{h \in S_i} E_hV$.
We show that $S_i = \lbrace 0,1,\ldots, i\rbrace$ for $0 \leq i \leq D$. By construction $S_0 = \lbrace 0\rbrace$ and $S_1 = \lbrace 0,1\rbrace$.
Using \eqref{eq:1im} we obtain $S_i \subseteq \lbrace 0,1,\ldots, i\rbrace$ for $0 \leq i \leq D$. By construction $S_{i-1} \subseteq S_i$
for $1 \leq i \leq D$. For $1 \leq i \leq D$ we have $S_{i-1}\not=S_i$; otherwise $P_{i-1} = P_i$, which forces $i\geq 2$ and $E_1 V \circ P_{i-1} \subseteq P_{i-1}$, which forces
$P_{i-1}=V$ by Theorem \ref{thm:Egen}, which forces  $S_{i-1}= \lbrace 0,1,\ldots, D\rbrace$, which contradicts $S_{i-1} \subseteq 
\lbrace 0,1,\ldots, i-1\rbrace\subseteq \lbrace 0,1,\ldots, D-1\rbrace$. By these comments $S_i = \lbrace 0,1,\ldots, i\rbrace$ for $0 \leq i \leq D$.
\\
\noindent ${\rm (iii)} \Rightarrow {\rm (i)}$ Let $0 \leq i \leq D$ and $0 \leq j \leq D-i$. We show that $q^{i+j}_{i,j} \not=0$ and
$q^h_{i,j}=0$ for $i+j<h\leq D$.
By \eqref{eq:expand},
\begin{align*}
(E_0V + E_1V+\cdots + E_iV)\circ (E_0V+E_1V+ \cdots + E_jV) = E_0V+E_1V+ \cdots +E_{i+j}V.
\end{align*}
By this and Theorem \ref{thm:fa}, we obtain $q^h_{i,j}=0$ for $i+j <h\leq D$. Also by Theorem \ref{thm:fa}, there exists
$0 \leq r \leq i$ and $0 \leq s \leq j$ such that $q^{i+j}_{r,s} \not=0$. By our above comments $i+j \leq r+s$. By construction
$0 \leq r \leq i$ and $0 \leq s \leq j$, so 
 $r=i$ and $s=j$. Therefore
$q^{i+j}_{i,j}\not=0$.
\end{proof}

\section{Irreducible $T$-modules and tridiagonal pairs}

Throughout this section, $\Gamma=(X, \mathcal R)$ denotes a distance-regular graph with diameter $D\geq 3$. We assume that
$\Gamma$ is $Q$-polynomial with respect to the ordering $\lbrace E_i \rbrace_{i=0}^D$ of the primitive idempotents. Fix $x \in X$
and write $T=T(x)$. We will show that $A$, $A^*$ act on each irreducible $T$-module as a tridiagonal pair.

\begin{lemma}\label{lem:TAA} The algebra $T$ is generated by $A, A^*$.
\end{lemma}
\begin{proof} The algebra $T$ is generated by  $M$ and $M^*$. Moreover $M$ is generated by $A$, and $M^*$ is generated by $A^*$.
\end{proof}

\noindent We review a few points:
\begin{enumerate}
\item[$\bullet$] the eigenspaces of $A$ are $\lbrace E_iV \rbrace_{i=0}^D$;
\item[$\bullet$] for $0 \leq i \leq D$, $\theta_i$ is the eigenvalue of $A$ for $E_iV$;
\item[$\bullet$] the eigenspaces of $A^*$ are $\lbrace E^*_iV \rbrace_{i=0}^D$;
\item[$\bullet$] for $0 \leq i \leq D$, $\theta^*_i$ is the eigenvalue of $A^*$ for $E^*_iV$;
\item[$\bullet$] for $0 \leq i \leq D$ we have
\begin{align*}
 A E^*_iV \subseteq E^*_{i-1}V + E^*_iV+E^*_{i+1}V,
 \end{align*}
 where $E^*_{-1}=0$ and $E^*_{D+1}=0$.
\end{enumerate}

 \begin{lemma} \label{lem:3TT} For $0 \leq i \leq D$ we have
\begin{align*}
 A^* E_iV \subseteq E_{i-1}V + E_iV+E_{i+1}V,
 \end{align*}
 where $E_{-1}=0$ and $E_{D+1}=0$.
\end{lemma}
\begin{proof} By the containment on the right in  \eqref{eq:tpr},
together with Definition \ref{def:qpoly}.
\end{proof}

\noindent  In Section 2 we mentioned that the standard module $V$ is an orthogonal direct sum of irreducible $T$-modules.
Let $W$ denote an irreducible $T$-module.
Observe that $W$ is an orthogonal direct sum of the nonzero subpaces among $\lbrace E^*_iW\rbrace_{i=0}^D$. Similarly
$W$ is an orthogonal direct sum 
 of the nonzero subpaces among $\lbrace E_iW\rbrace_{i=0}^D$. 
 
 \begin{lemma} \label{lem:WTT} Let $W$ denote an irreducible $T$-module. Then for $0 \leq i \leq D$ we have
\begin{align*}
&A E^*_iW \subseteq E^*_{i-1}W + E^*_iW + E^*_{i+1}W,
\\
&A^* E_i W \subseteq E_{i-1}W + E_iW+E_{i+1}W.
\end{align*}
\end{lemma}
\begin{proof} By Lemma \ref{lem:3TT}
and the comment above it.
\end{proof}

 \noindent Let $W$ denote an irreducible $T$-module.
By the {\it endpoint} of $W$ we mean
$\mbox{min}\lbrace i \vert 0\leq i \leq D, \; E^*_iW\not=0\rbrace $.
By the {\it diameter} of $W$ we mean
$ |\lbrace i \vert 0 \leq i \leq D,\; E^*_iW\not=0 \rbrace |-1 $.
By the {\it dual endpoint} of $W$ we mean
$\mbox{min}\lbrace i \vert 0\leq i \leq D, \; E_iW\not=0\rbrace $.
By
the {\it dual diameter} of $W$ we mean
$ |\lbrace i \vert 0 \leq i \leq D,\; E_iW\not=0 \rbrace |-1 $.

\begin{lemma} \label{lem:Wfacts}
{\rm \cite[Lemma 3.4, Lemma 3.9]{terwSub1}}
\label{lem:basic}
Let $W$ denote an irreducible $T$-module with endpoint $\rho$ and diameter $d$.
Then $\rho,d$ are nonnegative integers such that $\rho+d\leq D$.
Moreover the following {\rm (i), (ii)}  hold:
\begin{enumerate}
\item[\rm (i)] 
$E^*_iW \not=0$ if and only if $\rho \leq i \leq \rho+d$
$ \quad (0 \leq i \leq D)$;
\item[\rm (ii)]
$W = \sum_{i=\rho}^{\rho+d} E^*_{i}W \qquad (\mbox{orthogonal direct sum}). $
\end{enumerate}
\end{lemma}
\begin{proof} (i) By construction $E^*_\rho W\not=0$ and $E^*_iW=0$ for $0 \leq i < \rho$.
Suppose there exists an integer $i$ $(\rho < i \leq \rho+d)$ such that $E^*_iW=0$. Define $W'=E^*_\rho W + E^*_{\rho+1} W + \cdots + E^*_{i-1}W$.
By construction $0 \not= W' \subseteq  W$. Also by construction, $A^* W'\subseteq W'$. By Lemma \ref{lem:WTT} and $E^*_iW=0$ we obtain
$A W' \subseteq W'$. By these comments $W'$ is a $T$-module. We have $W=W'$ since the $T$-module $W$ is irreducible. This contradicts the
fact that $d$ is the diameter of $W$. We conclude that $E^*_iW\not=0$ for $\rho \leq i \leq \rho+d$. By the definition of the diameter $d$ we have
 $E^*_iW=0$ for $\rho+d < i \leq D$.
 \\
 \noindent (ii) By (i) and the comments above Lemma  \ref{lem:WTT}.
 \end{proof}

\begin{lemma} \label{lem:Wfacts2}
{\rm \cite[Lemma 3.4, Lemma 3.12]{terwSub1}}
\label{lem:basic2}
Let $W$ denote an irreducible $T$-module with 
dual endpoint $\tau$ and dual diameter $\delta$.
Then $\tau, \delta $ are nonnegative integers such that $\tau+\delta \leq D$. Moreover the following {\rm (i), (ii)} hold:
\begin{enumerate}
\item[\rm (i)] 
$E_iW \not=0$ if and only if $\tau \leq i \leq \tau+\delta$
$ \quad (0 \leq i \leq D)$;
\item[\rm (ii)]
$W = \sum_{i=\tau}^{\tau+\delta} E_{i}W \qquad (\mbox{orthogonal direct sum}). $
\end{enumerate}
\end{lemma}
\begin{proof} Similar to the proof of Lemma \ref{lem:Wfacts}.
\end{proof}

\begin{theorem} {\rm (See  \cite[Lemmas~3.9, 3.12]{terwSub1}.)} The pair $A, A^*$ acts on each irreducible $T$-module as a tridiagonal pair.
\end{theorem}
\begin{proof} By Definition \ref{def:tdp} and Lemmas \ref{lem:TAA}, \ref{lem:WTT}, \ref{lem:Wfacts}, \ref{lem:Wfacts2}.
\end{proof}

\begin{corollary} Let $W$ denote an irreducible $T$-module. The the diameter of $W$ is equal to the dual diameter of $W$.
\end{corollary}
\begin{proof} By the comment below Note  \ref{note:conv}.
\end{proof}

\noindent It is an ongoing project to describe the irreducible $T$-modules for a $Q$-polynomial distance-regular graph $\Gamma$. Comprehensive treatments can be found in
 \cite{bbit, itoOq, augIto, IKT, terwSub1, terwSub2, terwSub3, notesLS}.
In addition, there are papers about the thin condition \cite{cerzo, dickie1, dickie2, twistedG, terwSub3, zit1};
irreducible $T$-modules with endpoint one \cite{hobart};
$\Gamma$ being bipartite 
\cite{caugh2,curtin1, curtin2, lang3};
$\Gamma$ being almost-bipartite \cite{caugh5, lang4};
$\Gamma$ being dual bipartite 		
\cite{dickie4};
$\Gamma$ being almost dual bipartite \cite{dickieThesis};
$\Gamma$ being 2-homogeneous \cite{curtin3, curtin5, curtin6, nomTB};
$\Gamma$ being tight  \cite{aap12};
$\Gamma$ being a hypercube \cite{go};
$\Gamma$ being a Doob graph \cite{tanabe};
$\Gamma$ being a Johnson graph \cite{liang2, tan};
$\Gamma$ being a Grassmann graph \cite{liang1};
$\Gamma$ being a dual polar graph 
 \cite{boyd};
$\Gamma$ having a spin model in the Bose-Mesner algebra
\cite{curtin4, nomSM}.
Some miscellaneous topics about irreducible $T$-modules can be found in  \cite{qtetanddrg, itoTer3, TerPoly, lang2,jhl1,jhl2,aap1, sum1,sum2,ds, zit2}.

\section{Recurrent sequences}

 \noindent  In this section we have some comments about finite sequences that satisfy a 3-term recurrence.
 In later sections we will apply these comments to $Q$-polynomial distance-regular graphs.
 \medskip
 
 \noindent
 Throughout this section  fix an integer $D\geq 3$, and
let $\lbrace \theta_i\rbrace_{i=0}^D$ denote
 scalars in $\mathbb R$.

\begin{definition} \label{def:3term1}
\rm
Let $\beta, \gamma, \varrho$ denote scalars
in $\mathbb R$.
\begin{enumerate} 
\item[\rm (i)] The sequence 
  $\lbrace \theta_i\rbrace_{i=0}^D$ 
is said to be {\it recurrent} whenever $\theta_{i-1}\not=\theta_i$ for
$2 \leq i \leq D-1$, and 
\begin{equation}
\frac{\theta_{i-2}-\theta_{i+1}}{\theta_{i-1}-\theta_i} 
\label{eq:bp1}
\end{equation}
is independent of
$i$ for $2 \leq i \leq  D-1$.
\item[\rm (ii)] The sequence 
  $\lbrace \theta_i\rbrace_{i=0}^D$ 
is said to be {\it $\beta$-recurrent} whenever 
\begin{equation}
\theta_{i-2}-(\beta+1)\theta_{i-1}+(\beta +1)\theta_i -\theta_{i+1}
\label{eq:beta}
\end{equation}
is zero for 
$2 \leq i \leq D-1$.
\item[\rm (iii)] The sequence 
  $\lbrace \theta_i\rbrace_{i=0}^D$ 
is said to be {\it $(\beta,\gamma)$-recurrent} whenever 
\begin{equation}
\theta_{i-1}-\beta \theta_i+\theta_{i+1}=\gamma 
\label{eq:bg}
\end{equation}
 for 
$1 \leq i \leq D-1$.
\item[\rm (iv)] The sequence 
  $\lbrace \theta_i\rbrace_{i=0}^D$ 
is said to be {\it $(\beta,\gamma,\varrho)$-recurrent} whenever 
\begin{equation}
\theta^2_{i-1}-\beta \theta_{i-1}\theta_i+\theta^2_i 
-\gamma (\theta_{i-1} +\theta_i)=\varrho
\label{eq:vrt}
\end{equation}
 for 
$1 \leq i \leq D$.
\end{enumerate}
\end{definition}

\begin{lemma} 
\label{lem:recvsbrecS99}
The following are equivalent:
\begin{enumerate}
\item[\rm (i)] the sequence 
  $\lbrace \theta_i\rbrace_{i=0}^D$
is  recurrent;
\item[\rm (ii)]
the scalars $\theta_{i-1}\not=\theta_i$ for
$2 \leq i \leq D-1$, and 
there exists $\beta \in \mathbb R$ such that
  $\lbrace \theta_i\rbrace_{i=0}^D$ 
  is $\beta$-recurrent.
\end{enumerate}
\noindent Suppose {\rm (i), (ii)} hold. Then
the common value of \eqref{eq:bp1}
is equal to $\beta +1$.

\end{lemma}
\begin{proof} Routine.

\end{proof}

\begin{lemma}
\label{lem:brecvsbgrecS99}
For $\beta \in \mathbb R$ the following are equivalent:
\begin{enumerate}
\item[\rm (i)] the sequence 
  $\lbrace \theta_i\rbrace_{i=0}^D$
is  $\beta$-recurrent;
\item[\rm (ii)] there exists $\gamma \in \mathbb R$ such that
  $\lbrace \theta_i\rbrace_{i=0}^D$
  is $(\beta,\gamma)$-recurrent.
\end{enumerate}
\end{lemma}
\begin{proof}  Routine.
\end{proof}

\begin{lemma}
\label{lem:bgrecvsbgdrecS99}
The following {\rm (i), (ii)} hold
for all $\beta, \gamma \in \mathbb R$.
\begin{enumerate}
\item[\rm (i)]  Suppose 
  $\lbrace \theta_i\rbrace_{i=0}^D$
is  $(\beta,\gamma)$-recurrent. Then 
there exists $\varrho \in \mathbb R$ such that
  $\lbrace \theta_i\rbrace_{i=0}^D$
 is $(\beta,\gamma,\varrho)$-recurrent.
\item[\rm (ii)]  Suppose 
  $\lbrace \theta_i\rbrace_{i=0}^D$
is  $(\beta,\gamma,\varrho)$-recurrent, and that $\theta_{i-1}\not=\theta_{i+1}$
for $1 \leq i\leq D-1$. Then
  $\lbrace \theta_i\rbrace_{i=0}^D$
  is $(\beta,\gamma)$-recurrent.
\end{enumerate}
\end{lemma}

\begin{proof} 
Let  $p_i$ denote the expression on the left in \eqref{eq:vrt},
and observe
 \begin{align*}
p_i-p_{i+1} &= 
(\theta_{i-1}-\theta_{i+1})(\theta_{i-1}-\beta \theta_i +\theta_{i+1} - \gamma)
\end{align*}
for $1 \leq i \leq D-1$. 
Assertions (i), (ii) are both routine consequences of this.
\end{proof}

\noindent The following result is handy.

\begin{lemma} \label{lem:handy} Let $\beta, \gamma, \varrho \in \mathbb R$ and assume that $\lbrace \theta_i \rbrace_{i=0}^D$ is $(\beta, \gamma, \varrho)$-recurrent.
Then
\begin{align}
\label{eq:PP}
(2-\beta)\theta^2_i - 2 \gamma \theta_i - \varrho
=(\theta_i - \theta_{i-1})(\theta_i - \theta_{i+1})
 \qquad \qquad (0 \leq i \leq D),
\end{align}
where $\theta_{-1}$ and $\theta_{D+1}$ are defined by \eqref{eq:bg}  at $i=0$ and $i=D$.
\end{lemma}
\begin{proof} To verify \eqref{eq:PP} for $1 \leq i \leq D$, eliminate $\theta_{i+1}$ using   \eqref{eq:bg}, and evaluate the result
using \eqref{eq:vrt}.
To verify \eqref{eq:PP} for $0 \leq i \leq D-1$, eliminate $\theta_{i-1}$ using  \eqref{eq:bg}, and evaluate the result
using \eqref{eq:vrt}.
\end{proof}

\section{The tridiagonal relations}
Throughout this section $\Gamma=(X,\mathcal R)$ denotes a distance-regular graph with diameter $D\geq 3$. 
Let $\lbrace E_i \rbrace_{i=0}^D$ denote a $Q$-polynomial ordering of the primitive idempotents of $\Gamma$.
 Fix $x \in X$ and write $T=T(x)$. 
We will show that $A, A^*$ satisfy a pair of relations called the tridiagonal relations. We will also obtain
a recurrence satisfied by the eigenvalues $\lbrace \theta_i \rbrace_{i=0}^D$ and the dual eigenvalues $\lbrace \theta^*_i \rbrace_{i=0}^D$. We now state our two main results.
\begin{theorem} \label{thm:thmOne} {\rm (See \cite[Lemma~5.4]{terwSub3}.)} 
With the above notation, there exist real numbers $\beta$, $\gamma$, $\gamma^*$, $\varrho$, $\varrho^*$
such that
\begin{align}
\label{eq:TDone}
0 &= \lbrack A, A^2 A^* - \beta A A^* A + A^* A^2 - \gamma(AA^*+ A^*A) - \varrho A^* \rbrack,
\\
\label{eq:TDtwo}
0 &= \lbrack A^*, A^{*2} A - \beta A^* A A^* + A A^{*2} - \gamma^*(A^*A+ AA^*) - \varrho^* A \rbrack.
\end{align}
\end{theorem}
\begin{definition}\rm 
The relations \eqref{eq:TDone}, \eqref{eq:TDtwo} are called the {\it  tridiagonal relations}; see \cite{qSerre}.
\end{definition}

\begin{theorem}\label{thm:thmTwo} {\rm (See \cite[Proposition~3]{leonard} and \cite[Theorem~5.1]{bannai}.)} 
With the above notation, the scalars
\begin{align*}
\frac{\theta_{i-2}-\theta_{i+1}}{\theta_{i-1}-\theta_i}, \qquad \quad \frac{\theta^*_{i-2}-\theta^*_{i+1}}{\theta^*_{i-1}-\theta^*_i}
\end{align*}
are equal and independent of $i$ for $2 \leq i \leq D-1$.
\end{theorem}
\noindent We will prove Theorems \ref{thm:thmOne}, \ref{thm:thmTwo} shortly.

\begin{lemma}\label{lem:step1}
For $0 \leq i,j,r\leq D$ we have
\begin{enumerate}
\item[\rm (i)] 
$E^*_i A^{r} E^*_j =
\begin{cases}
0 & {\mbox{\rm if $r<|i-j|$}};\\
\not=0 &  {\mbox{\rm if $r=|i-j|$,
}}
\end{cases}$
\item[\rm (ii)] 
$E_i A^{*r} E_j =
\begin{cases}
0 & {\mbox{\rm if $r<|i-j|$}};\\
\not=0 &  {\mbox{\rm if $r=|i-j|$.
}}
\end{cases}$

\end{enumerate}
\end{lemma}
\begin{proof} (i) By Theorem \ref{lem:tpr}(i), $E^*_i A_r E^*_j=0$ if and only if $p^i_{r,j}=0$. The matrix $A_r$ is a polynomial in $A$ with degree exactly $r$.
The scalar $p^i_{r,j}$ is zero if $r< \vert i-j\vert$ and nonzero if $r=\vert i-j\vert$. The result follows.
\\
\noindent (ii) Similar to the proof of (i).
\end{proof}
\noindent We are going to prove a sequence of results. Each result has a `dual' obtained by interchanging $A$, $A^*$ and $E_i, E^*_i$ for $0 \leq i \leq D$.
We will not state each dual explicitly.

\begin{lemma}
\label{lem:step2}
For $0 \leq i,j,r,s\leq D$ we have
\begin{align*}
E^*_i A^r A^* A^s E^*_j = 
\begin{cases}
\theta^*_{j+s} E^*_i A^{r+s} E^*_j & \mbox{\rm if $r+s=i-j$};  \\
\theta^*_{j-s} E^*_i A^{r+s} E^*_j & \mbox{\rm if $r+s=j-i$};  \\
0  & \mbox{\rm if $r+s<\vert i-j\vert $}.
\end{cases}
\end{align*}
\end{lemma}
\begin{proof} Using $I= \sum_{h=0}^D  E^*_h$ and $A^*= \sum_{h=0}^D \theta^*_h E^*_h$ we find
\begin{align*}
E^*_i A^r I A^s E^*_j = \sum_{h=0}^D  E^*_i A^r E^*_h A^s E^*_j, \qquad 
E^*_i A^r A^* A^s E^*_j = \sum_{h=0}^D \theta^*_h E^*_i A^r E^*_h A^s E^*_j.
\end{align*}
In the above sums, evaluate each summand using Lemma \ref{lem:step1}(i). The result follows.
\end{proof}

\begin{lemma} \label{lem:step3}
For $0 \leq i \leq D-1$,
\begin{align*}
E_i A^* E_{i+1} - E_{i+1} A^* E_i = (E_0+E_1+ \cdots + E_i) A^*-A^* (E_0+E_1+\cdots + E_i).
\end{align*}
\end{lemma}
\begin{proof} Recall the convention $E_{-1}=0$. For $0 \leq j \leq D-1$ we have
\begin{align}
E_j A^* = E_j A^* (E_0+E_1+ \cdots + E_D)= E_j A^* (E_{j-1} + E_j + E_{j+1}),
\label{eq:sum1}
\end{align}
and similarly
\begin{align}
 A^*E_j = (E_{j-1} + E_j + E_{j+1}) A^* E_j.
\label{eq:sum2}
\end{align}
Sum \eqref{eq:sum1}, \eqref{eq:sum2} over $j=0,1,\ldots, i$. Take the difference between the two sums.
\end{proof}

\noindent Recall the Bose-Mesner algebra $M$ of $\Gamma$.

\begin{lemma} \label{lem:step4} 
We have
\begin{align*}
{\rm Span}\lbrace R A^* S - S A^* R \vert R, S \in M\rbrace = \lbrace YA^*-A^*Y \vert Y \in M\rbrace.
\end{align*}
\end{lemma} 
\begin{proof} Recall that $\lbrace E_i \rbrace_{i=0}^D$ is a basis for $M$. Note that $\lbrace E_0+\cdots +E_i\rbrace_{i=0}^D$ is a basis for $M$.
Observe that
\begin{align*}
&{\rm Span}\lbrace R A^* S - S A^* R \vert R, S \in M\rbrace 
\\
&={\rm Span}\lbrace E_i A^* E_j - E_j A^* E_i \vert 0 \leq i,j\leq D\rbrace 
\\
&={\rm Span}\lbrace E_i A^* E_{i+1} - E_{i+1} A^* E_i \vert 0 \leq i \leq D-1\rbrace 
\\
&={\rm Span}\lbrace (E_0+\cdots + E_i) A^* -  A^* (E_0 + \cdots +E_i) \vert 0 \leq i \leq D-1\rbrace 
\\
&={\rm Span}\lbrace (E_0+\cdots + E_i) A^* -  A^* (E_0 + \cdots +E_i) \vert 0 \leq i \leq D\rbrace 
\\
&= \lbrace YA^*-A^*Y \vert Y \in M\rbrace.
\end{align*}
\end{proof}
\noindent Let $\lambda, \mu$ denote commuting indeterminates. We define some polynomials $P(\lambda, \mu)$ and $P^*(\lambda, \mu)$ as follows.
Given real numbers $\beta, \gamma, \varrho$ define
\begin{align}
P(\lambda, \mu) = \lambda^2 - \beta \lambda \mu + \mu^2 - \gamma(\lambda + \mu) - \varrho.
\label{eq:P}
\end{align}
Given real numbers $\beta, \gamma^*, \varrho^*$ define
\begin{align}
P^*(\lambda, \mu) = \lambda^2 - \beta \lambda \mu + \mu^2 - \gamma^*(\lambda + \mu) - \varrho^*.
\label{eq:Ps}
\end{align}

\begin{lemma} \label{lem:step5} For real numbers $\beta, \gamma, \varrho$ the following are equivalent:
\begin{enumerate}
\item[\rm (i)] $0 = \lbrack A, A^2 A^* - \beta A A^* A + A^* A^2 - \gamma(AA^*+ A^*A) - \varrho A^* \rbrack$;
\item[\rm (ii)] $P(\theta_{i-1}, \theta_i)=0 $ for $1 \leq i \leq D$.
\end{enumerate}
\end{lemma}
\begin{proof} Let $C$ denote the expression on the right in (i). We have
\begin{align*}
C = (E_0 + \cdots + E_D)C (E_0 +\cdots + E_D) = \sum_{i=0}^D \sum_{j=0}^D E_i C E_j.
\end{align*}
For $0 \leq i,j\leq D$ use $E_iA=\theta_i E_i$ and $A E_j = \theta_j E_j$ to get
\begin{align} \label{eq:Cexpand}
E_i C E_j = E_i A^* E_j (\theta_i - \theta_j)P(\theta_i, \theta_j).
\end{align}
\noindent ${\rm (i)} \Rightarrow {\rm (ii)}$: For $1 \leq i \leq D$ we show $P(\theta_{i-1}, \theta_i)=0$.
We have $C=0$, so
\begin{align*}
0 = E_{i-1} C E_i = E_{i-1} A^* E_i (\theta_{i-1} - \theta_i)P(\theta_{i-1}, \theta_i).
\end{align*}
We have $E_{i-1}A^* E_i\not=0$ and $\theta_{i-1}-\theta_i\not=0$, so $P(\theta_{i-1}, \theta_i)=0$.
\\
\noindent ${\rm (ii)} \Rightarrow {\rm (i)}$: We show that $C=0$.
The polynomial $P(\lambda, \mu)$ is symmetric, so $P(\theta_i,\theta_{i-1})=0$ for $1 \leq i \leq D$.
To show that $C=0$, it suffices to show that $E_iCE_j=0$ for $0 \leq i,j\leq D$.
Let $i,j$ be given. We evaluate $E_iCE_j$ using \eqref{eq:Cexpand}.
 If $\vert i-j\vert > 1 $ then $E_iA^*E_j=0$, so $E_iCE_j=0$.
If $\vert i-j\vert =1$ then $P(\theta_i, \theta_j)=0$, so $E_iCE_j=0$.
If $i=j$ then $\theta_i-\theta_j=0$, so $E_iCE_j=0$.
In all cases $E_i C E_j=0$, so $C=0$.
\end{proof}

\noindent We are now ready to prove Theorem \ref{thm:thmOne}.
\medskip

\noindent {\it Proof of Theorem \ref{thm:thmOne}}. 
By Lemma  \ref{lem:step4} (with $R=A^2$ and $S=A$), there exists $Z \in M$ such that
\begin{align}
A^2 A^* A - A A^* A^2 = Z A^*-A^*Z.
\label{eq:AAZ}
\end{align}
The matrices $\lbrace A^i \rbrace_{i=0}^D$ form a basis for $M$, so there exists a polynomial $f$ with degree at most $D$ such that $Z=f(A)$.
Let $d$ denote the degree of $f$. We show that $d=3$. First assume that $d>3$. Multiply each term in \eqref{eq:AAZ} on the left by $E^*_d$
and the right by $E^*_0$. Evaluate the result using Lemmas \ref{lem:step1}, \ref{lem:step2} to get
\begin{align}\label{eq:3factors}
0 = c (\theta^*_0 - \theta^*_d) E^*_d A^d E^*_0,
\end{align}
where $c$ is the leading coefficient of $f$. By construction $c \not=0$. We have $\theta^*_0-\theta^*_d\not=0$  since $d \not=0$. Also
$E^*_d A^d E^*_0\not=0$ by Lemma \ref{lem:step1}. Therefore \eqref{eq:3factors} gives a contradiction. Next assume that $d<3$.
Multiply each term in \eqref{eq:AAZ} on the left by $E^*_3$ and on the right by $E^*_0$. Evaluate the result using Lemmas \ref{lem:step1}, \ref{lem:step2} to get
\begin{align}
(\theta^*_1-\theta^*_2) E^*_3 A^3 A^*_0=0.
\label{eq:3factorsL}
\end{align}
We have $\theta^*_1-\theta^*_2\not=0$. Also $E^*_3 A^3 E^*_0\not=0$ by Lemma \ref{lem:step1}. Therefore \eqref{eq:3factorsL} gives a contradiction.
We have shown that $d=3$. Define $\beta = c^{-1}-1$. Divide both sides of \eqref{eq:AAZ} by $c$, and evaluate the result using $d=3$ and
$c^{-1}=\beta+1$. We find that there exist $\gamma, \varrho \in \mathbb R$ such that
\begin{align*} 
(\beta+1)(A^2A^*A-AA^*A^2) = A^3A^*-A^*A^3 - \gamma(A^2 A^* - A^* A^2) -\varrho (A A^*-A^*A).
\end{align*} 
In this equation, we rearrange the terms to get
\begin{align}
0 &= \lbrack A, A^2 A^* - \beta A A^* A + A^* A^2 - \gamma(AA^*+ A^*A) - \varrho A^* \rbrack.
\label{eq:1TD}
\end{align}
This is
the first tridiagonal relation. To get the second tridiagonal relation, pick an integer $i$ $(2 \leq i \leq D-1)$. Multiply each term in
\eqref{eq:1TD} on the left by $E^*_{i-2}$ and on the right by $E^*_{i+1}$. Simplify the result using Lemma \ref{lem:step2} to get
\begin{align} \label{eq:twoterms}
0 = E^*_{i-2} A^3 E^*_{i+1} \bigl( \theta^*_{i-2} - (\beta+1) \theta^*_{i-1} + (\beta+1) \theta^*_i - \theta^*_{i+1}\bigr).
\end{align}
We have $E^*_{i-2}A^3 E^*_{i+1}\not=0$ by Lemma \ref{lem:step1}, so the coefficient in \eqref{eq:twoterms} must be zero. Therefore the sequence
$\lbrace \theta^*_i \rbrace_{i=0}^D$ is $\beta$-recurrent. By Lemma \ref{lem:brecvsbgrecS99} 
there exists $\gamma^* \in \mathbb R$ such that $\lbrace \theta^*_i \rbrace_{i=0}^D$ is $(\beta, \gamma^*)$-recurrent. By 
Lemma \ref{lem:bgrecvsbgdrecS99}, there exists $\varrho^* \in \mathbb R$ such that $\lbrace \theta^*_i \rbrace_{i=0}^D$ is $(\beta, \gamma^*, \varrho^*)$-recurrent. 
Consequently
$P^*(\theta^*_{i-1},\theta^*_i)=0$ for $1 \leq i \leq D$. Now $\beta$, $\gamma^*$, $\varrho^*$ satisfy \eqref{eq:TDtwo}  by Lemma  \ref{lem:step5}.
\hfill {$\Box$}

\begin{proposition} \label{prop:five} {\rm (See \cite[Theorem~4.3]{qSerre}.)} We refer to the scalars 
 $\beta, \gamma, \gamma^*,
\varrho, \varrho^*$ from Theorem \ref{thm:thmOne}.
\begin{enumerate}
\item[\rm (i)] 
The expressions
\begin{align*}
\frac{\theta_{i-2}-\theta_{i+1}}{\theta_{i-1}-\theta_i},\qquad \qquad  
 \frac{\theta^*_{i-2}-\theta^*_{i+1}}{\theta^*_{i-1}-\theta^*_i} 
\end{align*} 
 are both equal to $\beta +1$ for $2\leq i \leq D-1$;
 \item[\rm (ii)]   
$\gamma = \theta_{i-1}-\beta \theta_i + \theta_{i+1} $ $(1 \leq i \leq D-1)$;
\item[\rm (iii)] 
$\gamma^* = \theta^*_{i-1}-\beta \theta^*_i + \theta^*_{i+1}$ $(1 \leq i \leq D-1)$;
\item[\rm (iv)] 
$\varrho = \theta^2_{i-1}-\beta \theta_{i-1}\theta_i+\theta_i^2-\gamma (\theta_{i-1}+\theta_i)$ $(1 \leq i \leq D)$;
\item[\rm (v)] 
$\varrho^*= \theta^{*2}_{i-1}-\beta \theta^*_{i-1}\theta^*_i+\theta_i^{*2}-
\gamma^* (\theta^*_{i-1}+\theta^*_i)$ $(1 \leq i \leq D)$.
\end{enumerate}
\end{proposition}
\begin{proof} We start with item (iv).\\
\noindent (iv) By \eqref{eq:TDone} and Lemma \ref{lem:step5}.
\\
\noindent (v) By \eqref{eq:TDtwo} and Lemma \ref{lem:step5}.
\\
\noindent (ii) By Lemma \ref{lem:bgrecvsbgdrecS99} and (iv) above.
\\
\noindent (iii) By Lemma \ref{lem:bgrecvsbgdrecS99} and (v) above.
\\
\noindent (i) The sequence $\lbrace \theta_i \rbrace_{i=0}^D$ is $(\beta, \gamma)$-recurrent by (ii), so 
$\lbrace \theta_i \rbrace_{i=0}^D$ is $\beta$-recurrent. Similarly
$\lbrace \theta^*_i \rbrace_{i=0}^D$ is $\beta$-recurrent. The result follows.
\end{proof}

\noindent Theorem \ref{thm:thmTwo} is immediate from Proposition \ref{prop:five}(i).

\begin{remark} \rm The tridiagonal relations \eqref{eq:TDone}, \eqref{eq:TDtwo} are the defining relations
for the  tridiagonal algebra \cite[Definition~3.9]{qSerre}. 
Special cases of the tridiagonal algebra
include the universal enveloping algebra of the Onsager Lie algebra $O$ \cite[Remark~34.5]{aa}, the $q$-Onsager algebra $O_q$ \cite{pospart}, and the positive part  $U^+_q$ \cite[Remark~34.7]{aa} of the $q$-deformed enveloping algebra $U_q({\widehat{\mathfrak{sl}}}_2)$ \cite{charp}.
\end{remark}

\section{The primary $T$-module and the Askey-Wilson relations} 

\noindent Throughout this section $\Gamma=(X, \mathcal R)$ denotes a distance-regular
graph with diameter $D\geq 3$. 
Fix $x \in X$ and write $T=T(x)$. In Section 7 we described the primary $T$-module. In this section, we give more information 
under the assumption that $\Gamma$ is $Q$-polynomial. 
\medskip

\noindent
Throughout this section, we assume that $\Gamma$ is $Q$-polynomial with respect to the ordering $\lbrace E_i \rbrace_{i=0}^D$ of the
primitive idempotents.

\begin{lemma} \label{lem:LP} 
For the primary $T$-module $\mathcal V$ the following hold.
\begin{enumerate}
\item[\rm (i)] With respect to the basis $\lbrace {\bf 1}_i \rbrace_{i=0}^D$ the matrices representing $A$ and $A^*$ are
\begin{align*}
A : \quad  \begin{pmatrix} a_0 & b_0 &&&&{\bf 0}  \\ c_1 & a_1 & b_1 & && \\  &c_2 & \cdot  & \cdot && \\
&& \cdot & \cdot & \cdot & \\
&&& \cdot &\cdot & b_{D-1} \\
{\bf 0}&&&& c_D & a_D
 \end{pmatrix},
 \qquad \qquad A^*: \quad {\rm diag}( \theta^*_0, \theta^*_1, \ldots, \theta^*_D).
\end{align*}
\item[\rm (ii)] 
With respect to the basis $\lbrace {\bf 1}^*_i \rbrace_{i=0}^D$ the matrices representing $A$ and $A^*$ are
\begin{align*}
A: \quad {\rm diag}( \theta_0, \theta_1, \ldots, \theta_D), \qquad \qquad
A^* : \quad  \begin{pmatrix} a^*_0 & b^*_0 &&&&{\bf 0}  \\ c^*_1 & a^*_1 & b^*_1 & && \\  &c^*_2 & \cdot  & \cdot && \\
&& \cdot & \cdot & \cdot & \\
&&& \cdot &\cdot & b^*_{D-1} \\
{\bf 0}&&&& c^*_D & a^*_D
 \end{pmatrix}.
\end{align*}
\end{enumerate}
\end{lemma}
\begin{proof} By parts (ii), (iv) of Lemmas \ref{lem:Taction}, \ref{lem:Taction2}. 
\end{proof}

\begin{lemma} \label{lem:PMLP} The pair $A, A^*$ acts on the primary $T$-module $\mathcal V$ as a Leonard pair.
\end{lemma}
\begin{proof} By Definition \ref{def:lp} and Lemma \ref{lem:LP}.
\end{proof}
\noindent  In Section 15 we saw that $A, A^*$ satisfy the tridiagonal relations. In the literature, there is  another pair of relations called the Askey-Wilson relations \cite{zhedhidden},
that resemble the tridiagonal relations but are more elementary. It is shown in \cite[Theorem~1.5]{vidunas} that any Leonard pair satisfies the Askey-Wilson 
relations. This and Lemma \ref{lem:PMLP} imply that $A, A^*$ satisfy the Askey-Wilson relations on the primary $T$-module $\mathcal V$. 
We will show directly, that on the primary $T$-module $\mathcal V$ the $A, A^*$ satisfy
 the Askey-Wilson relations.
 \medskip
 
 \noindent We have a comment about Lemma \ref{lem:step1}. In that lemma there are some inequalities. We will need the fact that these
 inequalities hold on the primary $T$-module $\mathcal V$.

\begin{lemma}\label{lem:nonz} 
For $0 \leq i ,j\leq D$ and $r=\vert i-j\vert$ the following hold on the primary $T$-module $\mathcal V$:
\begin{align*}
E^*_i A^r E^*_j \not=0, \qquad \qquad E_i A^{*r} E_j \not=0.
\end{align*}
\end{lemma}
\begin{proof} Use the matrix representations in Lemma  \ref{lem:LP}, together with the fact that $b_{i-1} c_i\not=0$ and $b^*_{i-1} c^*_i \not=0$ for
$1 \leq i \leq D$.
\end{proof}

\noindent 
We now display the Askey-Wilson relations. Recall the scalars $\beta, \gamma, \gamma^*, \varrho, \varrho^*$ from Theorem \ref{thm:thmOne}.
The following result is a variation on \cite[Theorem~1.5]{vidunas}.
\begin{theorem} \label{thm:AW} {\rm (See \cite[Theorem~1.5]{vidunas}.)}
There exist real numbers $\omega, \eta, \eta^*$ such that the following hold on the primary $T$-module $\mathcal V$:
\begin{align}
A^2 A^* - \beta AA^*A+ A^*A^2 - \gamma(AA^*+ A^*A)-\varrho A^*&= \gamma^* A^2 + \omega A + \eta I, \label{eq:AW1}
\\
 A^{*2} A - \beta A^* A A^* + A A^{*2} - \gamma^*(A^* A+A A^*) - \varrho^* A &= \gamma A^{*2} + \omega A^* + \eta^*I. \label{eq:AW2}
\end{align}
\end{theorem}
\begin{proof} Let  $\mathcal L$ denote the expression on the left in \eqref{eq:AW1}. By the tridiagonal relation \eqref{eq:TDone}, we see that
 $\lbrack A, \mathcal L\rbrack=0$
on $\mathcal V$. The restriction of $A$ to $\mathcal V$ is diagonalizable and has all eigenspaces of dimension one.
Recall from linear algebra that for a diagonalizable linear transformation $\sigma$ that  has all eigenspaces of dimension one, any linear transformation that
commutes with $\sigma$ must be a polynomial in $\sigma$. 
Therefore there exists a polynomial $f$ such that $\mathcal L = f(A)$ on $\mathcal V$.
The minimal polynomial of $A$ on $\mathcal V$ has degree $D+1$, so we may choose $f$ such that its degree is at most $D$. Let $d$ denote
the degree of $f$.
We show that $d\leq 2$. Suppose that $d\geq 3$. In the equation $\mathcal L = f(A)$, multiply each term on the left by $E^*_d$ and on the right by $E^*_0$.
We have $E^*_d \mathcal L E^*_0=0$ by Lemma \ref{lem:step2}, so 
 $0 = E^*_d f(A) E^*_0$ on $\mathcal V$. By Lemma \ref{lem:step1} we have $E^*_d f(A) E^*_0= \alpha E^*_d A^d E^*_0$, where $\alpha$ is the leading coefficient of $f$.
By construction $\alpha\not=0$. By Lemma \ref{lem:nonz} we have
 $E^*_d A^d E^*_0\not=0$ on $\mathcal V$. This is a contradiction, so $d\leq 2$.
There exist real numbers $\varepsilon, \omega, \eta$ such that $\mathcal L = \varepsilon A^2 + \omega A + \eta I$ on $\mathcal V$. We show that $\gamma^*=\varepsilon $.
Note that on $\mathcal V$,
\begin{align*}
\gamma^* E^*_2 A^2 E^*_0 = (\theta^*_0 - \beta \theta^*_1 + \theta^*_2)E^*_2 A^2 E^*_0  =  E^*_2 \mathcal L E^*_0 = E^*_2 \bigl( \varepsilon A^2 + \omega A + \eta I\bigr) E^*_0 = \varepsilon E^*_2 A^2 E^*_0.
\end{align*}
By  Lemma \ref{lem:nonz} we have $E^*_2 A^2 E^*_0\not=0$ on $\mathcal V$, so $\gamma^*=\varepsilon $. We have shown that \eqref{eq:AW1} holds on $\mathcal V$.
Interchanging the roles of $A, A^*$ in the argument so far, we see that there
exist real numbers $\omega^*, \eta^*$ such that on $\mathcal V$,
\begin{align} \label{eq:AW2p}
A^{*2} A - \beta A^* A A^* + A A^{*2} - \gamma^*(A^*A+AA^*) - \varrho^* A = \gamma A^{*2} + \omega^* A^* + \eta^* I.
\end{align}
We show that $\omega = \omega^*$. Take the commutator of \eqref{eq:AW1} with $A^*$. This shows that on $\mathcal V$,
\begin{align*} 
A^2 A^{*2} -&\beta A A^* A A^* + \beta A^* A A^* A - A^{* 2} A^2 - \gamma(A A^{*2}-A^{*2} A)\\
&= \gamma^*(A^2 A^* - A^* A^2)+\omega(A A^*-A^* A).
\end{align*}
\noindent Next, take the commutator of \eqref{eq:AW2p} with $A$. This shows that on $\mathcal V$,
\begin{align*} 
 A^{*2}A^2 -&\beta  A^* A A^*A + \beta A A^* A A^*  - A^2A^{* 2}  - \gamma^*(A^* A^{2}-A^{2} A^*)\\
&= \gamma(A^{*2} A - A A^{*2})+\omega^*( A^*A-AA^*).
\end{align*}
Adding the above two equations, we find that on $\mathcal V$,
\begin{align*}
0 = (\omega-\omega^*) \bigl(AA^*-A^*A).
\end{align*}
We have $AA^*\not=A^*A$ on $\mathcal V$, because the $T$-module $\mathcal V$ is irreducible. Therefore $\omega=\omega^*$.
We have shown that \eqref{eq:AW2} holds on $\mathcal V$. 
\end{proof}

\begin{definition}\label{def:AW}\rm (See \cite{zhedhidden}.)
The relations \eqref{eq:AW1}, \eqref{eq:AW2} are called the {\it Askey-Wilson relations}.
\end{definition}

\noindent Next we consider how to compute the scalars $\omega, \eta, \eta^*$ from Theorem \ref{thm:AW}. To facilitate this computation,
we bring in some notation.
Recall from Proposition \ref{prop:five}
that
\begin{align}
\gamma= \theta_{i-1}-\beta \theta_i + \theta_{i+1}  \qquad \quad (1 \leq i \leq D-1), \label{eq:gamc}\\
\gamma^*= \theta^*_{i-1}-\beta \theta^*_i + \theta^*_{i+1}  \qquad \quad (1 \leq i \leq D-1). \label{eq:gamcs}
\end{align}
\begin{definition}\rm Define the real numbers
\begin{align*}
\theta_{-1}, \quad \theta_{D+1}, \quad \theta^*_{-1}, \quad \theta^*_{D+1}
\end{align*}
such that \eqref{eq:gamc}, \eqref{eq:gamcs} hold at $i=0$ and $i=D$.
\end{definition}
\noindent Recall the polynomials $P$, $P^*$ from \eqref{eq:P},
\eqref{eq:Ps}.

\begin{lemma} \label{lem:extend} 
The following hold for $0 \leq i \leq D$:
\begin{enumerate}
\item[\rm (i)] $P(\theta_i, \theta_i) = (\theta_i - \theta_{i-1})(\theta_i - \theta_{i+1})$;
\item[\rm (ii)] $P^*(\theta^*_i, \theta^*_i) = (\theta^*_i - \theta^*_{i-1})(\theta^*_i - \theta^*_{i+1})$.
\end{enumerate}
\end{lemma}
\begin{proof}  (i) By Lemma \ref{lem:handy} and since $\lbrace \theta_i \rbrace_{i=0}^D$ is $(\beta, \gamma, \varrho)$-recurrent.
\\
\noindent (ii) Similar to the proof of (i).
\end{proof}

\begin{proposition} \label{prop:wee} {\rm (See \cite[Theorem~5.3]{vidunas}.)}
With the above notation, we have
\begin{enumerate}
\item[\rm (i)] $\omega= a^*_i (\theta_i - \theta_{i+1}) + a^*_{i-1} (\theta_{i-1}-\theta_{i-2}) - \gamma^*(\theta_{i-1} + \theta_i)
\qquad (1 \leq i \leq D)$,
\item[\rm (ii)] $\omega= a_i (\theta^*_i - \theta^*_{i+1}) + a_{i-1} (\theta^*_{i-1}-\theta^*_{i-2}) - \gamma(\theta^*_{i-1} + \theta^*_i)
\qquad (1 \leq i \leq D)$,
\item[\rm (iii)] $\eta = a^*_i (\theta_i-\theta_{i-1})(\theta_i - \theta_{i+1}) -\omega \theta_i - \gamma^* \theta^2_i \qquad (0 \leq i \leq D)$,
\item[\rm (iv)] $\eta^* = a_i (\theta^*_i-\theta^*_{i-1})(\theta^*_i - \theta^*_{i+1}) -\omega \theta^*_i - \gamma \theta^{*2}_i \qquad (0 \leq i \leq D)$.
\end{enumerate}
\end{proposition} 
\begin{proof} We start with item (iii). \\
\noindent (iii) In the Askey-Wilson relation \eqref{eq:AW1}, multiply each term on the left by $E_i$ and on the right by $E_i$. This shows that on $\mathcal V$,
\begin{align*}
E_i A^*E_i P(\theta_i,\theta_i) = E_i \bigl( \gamma^* \theta^2_i + \omega \theta_i + \eta \bigr).
\end{align*}
By Lemma \ref{lem:eae} we have $E_i A^* E_i = a^*_i E_i$ on $\mathcal V$. By Lemma  \ref{lem:PB}(iii), $E_i\not=0$ on $\mathcal V$. By these comments,
\begin{align*}
a^*_i P(\theta_i,\theta_i) =  \gamma^* \theta^2_i + \omega \theta_i + \eta.
\end{align*}
Solve this equation for $\eta$, and evaluate the result using Lemma \ref{lem:extend}(i).
\\
\noindent (iv) Similar to the proof of (iii).
\\
\noindent (i) Subtract (iii) (at $i$) from (iii) (at $i-1$).
\\
\noindent (ii) Similar to the proof of (i).
\end{proof}

\section{ The Pascasio characterization of the $Q$-polynomial property}

\noindent In \cite{aap3} Pascasio characterized the $Q$-polynomial distance-regular graphs using the dual eigenvalues $\theta^*_i$  and the
intersection numbers $a_i$. In this section we give a proof of her result that uses some ideas of Hanson \cite{hanson1}.
\medskip

\noindent
Throughout this section  $\Gamma=(X, \mathcal R)$ denotes a distance-regular graph with diameter $D\geq 3$. 

\begin{definition}\rm 
\label{def:Qpoly}
Let $E$ denote a nontrivial primitive idempotent of $\Gamma$. We say that $\Gamma$ is {\it $Q$-polynomial with respect to $E$} whenever
there exists a $Q$-polynomial ordering $\lbrace E_i \rbrace_{i=0}^D$ of the primitive idempotents of $\Gamma$ such that $E=E_1$.
\end{definition}

\begin{theorem} \label{thm:Pasc} {\rm  (See \cite[Theorem~1.2]{aap3}.)}
Let $E=\vert X \vert^{-1} \sum_{i=0}^D \theta^*_i A_i$ denote a nontrivial primitive idempotent of $\Gamma$. Then $\Gamma$ is
$Q$-polynomial with respect to $E$ if and only if the following conditions hold:
\begin{enumerate}
\item[\rm (i)] $\theta^*_i \not=\theta^*_0$ $(1 \leq i \leq D)$;
\item[\rm (ii)] there exist real numbers $\beta, \gamma^*$ such that
\begin{align}
    \theta^*_{i-1} - \beta \theta^*_i + \theta^*_{i+1} = \gamma^* \qquad \qquad (1 \leq i \leq D-1);
    \label{eq:3T}
    \end{align}
 \item[\rm (iii)] there exist real numbers $\gamma, \omega, \eta^*$ such that
 \begin{align*}
 a_i (\theta^*_i -\theta^*_{i-1})(\theta^*_i - \theta^*_{i+1}) = \gamma \theta^{*2}_i + \omega \theta^*_i + \eta^*
 \qquad \quad (0 \leq i \leq D),
 \end{align*}
 where $\theta^*_{-1}$, $\theta^*_{D+1}$ are defined such that \eqref{eq:3T} holds at $i=0$ and $i=D$.
 \end{enumerate}
 \end{theorem}
\noindent We will prove Theorem \ref{thm:Pasc} shortly. In the meantime, let $\lbrace E_i \rbrace_{i=0}^D$ denote an ordering of the primitive idempotents of $\Gamma$, and abbreviate $E=E_1$. Let $x \in X$ and write $T=T(x)$, $A^*=A^*_1$.
\medskip

\begin{definition}\label{def:delta} \rm With the above notation, 
define a graph $\Delta_E$ with vertex set $\lbrace 0,1,\ldots, D\rbrace$ such that for $0 \leq i,j\leq D$, the vertices $i,j$ are adjacent  whenever
$i\not=j$ and  $q^1_{i,j}\not=0$.
\end{definition}

\begin{lemma} \label{lem:diag1} For the graph $\Delta_E$ in Definition \ref{def:delta}, the vertex $0$ is adjacent to vertex $1$ and no other vertex.
\end{lemma}
\begin{proof} We have $q^1_{0,j} = \delta_{1,j}$ for $0 \leq j \leq D$.
\end{proof}

\begin{lemma} \label{lem:3V} For distinct vertices $i,j$ of $\Delta_E$  the following are equivalent:
\begin{enumerate}
\item[\rm (i)]  $i,j$ are adjacent in $\Delta_E$;
\item[\rm (ii)]  $E_i A^* E_j \not=0$;
\item[\rm (iii)]  $E_i A^* E_j \not=0$ on the primary $T$-module $\mathcal V$.
\end{enumerate}
\end{lemma}
\begin{proof} By the triple product relations and Lemma \ref{lem:tprP}.
\end{proof}

\begin{lemma} \label{lem:diag2} For the graph $\Delta_E$ in Definition \ref{def:delta}, assume that $E$ is nondegenerate. Then $\Delta_E$  is connected.
\end{lemma}
\begin{proof} Let $S \subseteq \lbrace 0,1,\ldots, D\rbrace$ denote the connected component of $\Delta_E$ that contains $0,1$. We show that
$S =\lbrace 0,1,\ldots, D\rbrace$. 
Define $U=\sum_{i \in S} E_iV$. By construction $E_0V\subseteq U$ and $EV \subseteq U$. By Lemma \ref{lem:mq} and Theorem \ref{thm:fa}, the following holds
for $0 \leq i \leq D$:
\begin{align*}
EV \circ E_i V = \sum_{\stackrel{ \scriptstyle 0 \leq h \leq D }{ \scriptstyle q^1_{i,h} \not=0}} E_hV.
\end{align*}
Therefore $EV \circ U \subseteq U$. We assume that $E$ is nondegenerate, so the function algebra $V$ is generated by $EV$. By these
comments  $U=V$, so $S = \lbrace 0,1,\ldots, D\rbrace$.
\end{proof}

\noindent {\it Proof of Theorem \ref{thm:Pasc}}. First we assume that $\Gamma$ is $Q$-polynomial with respect to $E$. We saw earlier that
$E$ satisfies (i)--(iii).
Next we assume that $E$ satisfies (i)--(iii), and show that $\Gamma$ is $Q$-polynomial with respect to $E$. 
Fix $x \in X$ and write $T=T(x)$. Abbreviate $E=E_1$ and $A^*=A^*_1$. For the time being, let $\lbrace E_i \rbrace_{i=2}^D$ denote any ordering
of the remaining nontrivial primitive idempotents of $\Gamma$. For $0 \leq i \leq D$ let  $\theta_i$ denote the eigenvalue of $\Gamma$ for $E_i$.
Consider the graph $\Delta_E$ from Definition \ref{def:delta}. By Lemma \ref{lem:diag1}, in $\Delta_E$ the vertex $0$ is adjacent to vertex $1$ and
no other vertices. Note that $E$ is nondegenerate by condition (i) in the theorem statement, so the graph $\Delta_E$ is connected in view of Lemma \ref{lem:diag2}.
 We will show that $\Delta_E$ is a path.
By \eqref{eq:3T} the sequence $\lbrace \theta^*_i \rbrace_{i=0}^{D}$ is $(\beta, \gamma^*)$-recurrent. By Lemma \ref{lem:bgrecvsbgdrecS99}
 there exists $\varrho^* \in \mathbb R$
such that $\lbrace \theta^*_i \rbrace_{i=0}^{D}$ is $(\beta, \gamma^*, \varrho^*)$-recurrent. By this and Lemma \ref{lem:handy} we obtain
\begin{align*}
P^*(\theta^*_i, \theta^*_i) = (\theta^*_i - \theta^*_{i-1})(\theta^*_i - \theta^*_{i+1}) \qquad \qquad (0 \leq i \leq D),
\end{align*}
where the polynomial $P^*$ is from \eqref{eq:Ps}.
\\
\noindent Claim 1. On the primary $T$-module $\mathcal V$,
\begin{align}
 A^{*2} A - \beta A^* A A^* + A A^{*2} - \gamma^*(A^* A+A A^*) - \varrho^* A &= \gamma A^{*2} + \omega A^* + \eta^*I. \label{eq:PAW2}
 \end{align}
 \noindent Proof of Claim 1. Let $\mathcal L$ denote the left-hand side of \eqref{eq:PAW2}. On the $T$-module $\mathcal V$,
 \begin{align*}
 \mathcal L &= \sum_{i=0}^D \sum_{j=0}^D E^*_i \mathcal L E^*_j 
 \\
&= \sum_{i=0}^D \sum_{j=0}^D E^*_i A E^*_j P^*(\theta^*_i, \theta^*_j)
\\
&= \sum_{i=0}^D E^*_i A E^*_i P^*(\theta^*_i, \theta^*_i)
\\
&= \sum_{i=0}^D E^*_i  a_i (\theta^*_i-\theta^*_{i-1})(\theta^*_i - \theta^*_{i+1})
\\
&= \sum_{i=0}^D E^*_i \bigl(\gamma \theta^{*2}_i +\omega \theta^*_i + \eta^*\bigr)
\\
&= \gamma A^{*2} + \omega A^* + \eta^* I.
 \end{align*}
We have shown \eqref{eq:PAW2}.
\\
\noindent Claim 2. Let $i,j$ denote vertices in $\Delta_E$ that are at distance $\partial(i,j)=2$. Assume that there exists a unique vertex $h$ in $\Delta_E$
that is adjacent to both $i$ and $j$. Then 
$\gamma = \theta_i - \beta \theta_h + \theta_j$.
\\
\noindent Proof of Claim 2. In the equation \eqref{eq:PAW2}, multiply each term on the left by $E_i$ and on the right by $E_j$.
Simplify the result using Lemma \ref{lem:3V}. To aid this simplification, note that
\begin{align*} 
E_i A^{*2} E_j = E_i A^* \Biggl( \sum_{r=0}^D E_r \Biggr) A^* E_j = E_i A^*E_h A^* E_j
\end{align*}
and 
\begin{align*} 
E_i A^*   A A^*  E_j = E_i A^* \Biggl( \sum_{r=0}^D \theta_r E_r \Biggr) A^* E_j =\theta_h E_i A^*E_h A^* E_j.
\end{align*}
\noindent By these comments, the following holds on the $T$-module $\mathcal V$:
\begin{align*}
 0 = E_i A^* E_h A^* E_j \bigl( \theta_i - \beta \theta_h + \theta_j - \gamma\bigr).
 \end{align*}
 We show that $ E_i A^* E_h A^* E_j \not=0$ on $\mathcal V$. By  Lemma \ref{lem:3V} and the construction,  $ E_i A^* E_h$ and $E_h A^* E_j$ are nonzero on $\mathcal V$.
 The dimension of $E_h \mathcal V$ is one, so $E_h A^* E_j \mathcal V = E_h\mathcal V$. By these comments $ E_i A^* E_h A^* E_j \mathcal V = E_i A^* E_h\mathcal V \not=0$.
 We have shown that $ E_i A^* E_h A^* E_j \not=0$ on $\mathcal V$,
 so 
$\gamma = \theta_i - \beta \theta_h + \theta_j$. Claim 2 is proved.
\\
\noindent We can now easily show that $\Delta_E$ is a path. Since $\Delta_E$ is connected, and since vertex $0$ is adjacent only to vertex $1$, 
it suffices to show that each vertex in $\Delta_E$ is adjacent to at most two other vertices in $\Delta_E$. Suppose there exists a vertex $i$ of $\Delta_E$
that is adjacent to at least three vertices in $\Delta_E$. Of all such vertices, pick $i$ such that $\partial(0,i)$ is minimal. Without loss of generality,
we may assume that the vertices of $\Delta_E$ are labelled such that $\partial(0,i)=i$, and vertices $0,1,2,\ldots, i$ form a path in $\Delta_E$.
By construction $i\geq 1$. By assumption, there exist distinct vertices $j, j'$ in $\Delta_E$ that are adjacent to $i$ and not equal to $i-1$.
By construction, $\partial(i-1,j)=2$ and $i$ is the unique vertex in $\Delta_E$ that is adjacent to both $i-1, j$. By Claim 2,
 $\gamma = \theta_{i-1}-\beta \theta_i + \theta_j$. Repeating the argument with $j$ replaced by $j'$, we obtain
$\gamma = \theta_{i-1}-\beta \theta_i + \theta_{j'}$. By these comments $\theta_j = \theta_{j'}$ for a contradiction. We conclude that $\Delta_E$
is a path.  Relabelling $\lbrace E_i \rbrace_{i=2}^D$ if necessary, we may assume without loss of generality that 
vertices $i-1$ and $i$ are adjacent in $\Delta_E$ for $1 \leq i \leq D$. The ordering $\lbrace E_i \rbrace_{i=0}^D$ is $Q$-polynomial by
Theorem \ref{thm:faq}, because  item (ii) of that theorem is satisfied by $E=E_1$.
By these comments and Definition \ref{def:Qpoly}, the graph $\Gamma$ is $Q$-polynomial with respect to $E$.
\hfill {$\Box$}

\begin{note}\label{note:rec} \rm Referring to Theorem \ref{thm:Pasc}, assume that $\Gamma$ is $Q$-polynomial with respect to $E$.
For this $Q$-polynomial structure the eigenvalue sequence
  $\lbrace \theta_i \rbrace_{i=0}^D$
 is obtained as follows:
\begin{enumerate}
\item[$\bullet$] $\theta_0$ is the valency $k$ of $\Gamma$;
\item[$\bullet$] $\theta_1=k \theta^*_1/\theta^*_0$ by Lemma \ref{lem:recognize};
\item[$\bullet$] $\theta_2, \theta_3, \ldots, \theta_D$ are recursively found using
\begin{align*}
\theta_{i-1} - \beta \theta_i + \theta_{i+1} = \gamma \qquad \qquad (1 \leq i \leq D-1),
\end{align*}
where $\beta$, $\gamma$ are from Theorem \ref{thm:Pasc}.
\end{enumerate}
\end{note}

\begin{note}\rm A variation on Theorem \ref{thm:Pasc} is given in \cite{jtz}.
\end{note}

\section{Distance-regular graphs with classical parameters}
In \cite[Section~6.1]{bcn} Brouwer, Cohen, and Neumaier introduce a  type of distance-regular graph, said to have classical parameters.
In \cite[Section~8.4]{bcn} they show that these graphs are $Q$-polynomial. In this section we give a short proof of this fact, using
the Pascasio characterization from Theorem \ref{thm:Pasc}.
\medskip

\noindent Throughout this section $\Gamma=(X, \mathcal R)$ denotes a distance-regular graph with diameter $D\geq 3$.
Let $k$ denote the valency of $\Gamma$.
\medskip

\noindent We now recall what it means for $\Gamma$ to have classical parameters. We will use the following notation.
For a nonzero integer $b$ define
\begin{align*}
\left[  \begin{matrix}  i \\  1 \end{matrix} \right] 
=\left[  \begin{matrix} i \\ 1 \end{matrix} \right]_b=1+b+b^2+ \cdots + b^{i-1}. 
\end{align*}

\begin{definition} \label{def:cp} \rm (See \cite[p.~193]{bcn}.) The graph $\Gamma$ has {\it classical parameters} $(D, b, \alpha, \sigma)$ whenever the intersection numbers
satisfy
\begin{align*}
 c_i & = \left[  \begin{matrix} i \\ 1 \end{matrix} \right]  \left( 1 + \alpha \left[  \begin{matrix} i-1 \\ 1 \end{matrix} \right] \right)
 \qquad \quad (0 \leq i \leq D),
 \\
b_i &=\left( \left[  \begin{matrix} D \\ 1 \end{matrix} \right] -
 \left[  \begin{matrix} i \\ 1 \end{matrix} \right] \right)
\left(\sigma - \alpha \left[  \begin{matrix} i \\ 1 \end{matrix} \right] \right)
\qquad \qquad (0 \leq i \leq D).
 \end{align*}
 \end{definition}
 
 \begin{theorem} \label{thm:BCN} {\rm (See \cite[Section~8.4]{bcn}.)}
 Assume that $\Gamma$ has classical parameters $(D, b, \alpha, \sigma)$. Then the following {\rm (i)--(iv)} hold.
 \begin{enumerate}
 \item[\rm (i)] $\theta=\frac{b_1}{b} -1$ is an eigenvalue of $\Gamma$.
  \item[\rm (ii)] Let $E= \vert X \vert^{-1} \sum_{i=0}^D \theta^*_i A_i$ denote the associated primitive idempotent. Then
  \begin{align*}
  \frac{\theta^*_i}{\theta^*_0} = 1 + \left(\frac{\theta}{k}-1\right)  \left[  \begin{matrix} i \\ 1 \end{matrix} \right] b^{1-i} \qquad \quad (0 \leq i \leq D).
  \end{align*}
   \item[\rm (iii)] $\theta\not=k$.
   \item[\rm (iv)] $\Gamma$ is $Q$-polynomial with respect to $E$.
 \end{enumerate}
\end{theorem}
\begin{proof} (i), (ii) Apply Lemma \ref{lem:recognize} with $\theta= \frac{b_1}{b} -1$ and
 \begin{align*}
  \sigma_i = 1 + \left(\frac{\theta}{k}-1\right)  \left[  \begin{matrix} i \\ 1 \end{matrix} \right] b^{1-i} \qquad \quad (0 \leq i \leq D).
  \end{align*}
\\
\noindent (iii) Suppose $\theta=k$. We have $b_1 = b(k+1)$, so $b >0$. We have $b\geq 1$ since $b$ is an integer.
Therefore $b_1\geq k+1$, a contradiction.
We have shown that $\theta\not=k$.
\\
\noindent (iv) The conditions of Theorem \ref{thm:Pasc} are satisfied using $\beta = b + b^{-1}$ and
\begin{align*}
\gamma &= \frac{ \alpha (b^D+1) + \sigma (b-1) + 1-b}{b},
\\
\gamma^* &= \theta^*_0 \,\frac{\alpha (b^D-b) + \sigma (b-1)+ b^2-b}{k b},
\\
\omega &= \Psi (\theta^*_1-\theta^*_0)- 2 \gamma \theta^*_0,
\\
\eta^* &= \gamma \theta^{*2}_0 - \Psi \theta^*_0 (\theta^*_1-\theta^*_0),
\end{align*}
\noindent where
\begin{align*}
\Psi = 1 - \sigma - \frac{\alpha}{b} \left[  \begin{matrix} D+1\\ 1 \end{matrix} \right] .
\end{align*}
\end{proof}

\begin{lemma} \label{lem:eig} {\rm (See \cite[Corollary~8.4.2]{bcn}.)} Assume that $\Gamma$ has classical parameters $(D, b, \alpha, \sigma)$. For the $Q$-polynomial
structure in Theorem \ref{thm:BCN}, the eigenvalue sequence is 
\begin{align*}
\theta_i = \frac{b_i}{b^i}- \left[  \begin{matrix} i\\ 1 \end{matrix} \right] \qquad \qquad (0 \leq i \leq D).
\end{align*}
\end{lemma}
\begin{proof}  Routine calculation using Note \ref{note:rec}.
\end{proof}

\section{The balanced set characterization of the $Q$-polynomial property }
In this section we give a characterization of the $Q$-polynomial property, known as the balanced set condition.
The result is given in Theorem \ref{thm:bsc} below. The result first appeared in \cite{QPchar}. More recent versions can be found in
 \cite{bcn,newineq, suz, bbit}.
\medskip

\noindent
Throughout this section  $\Gamma=(X, \mathcal R)$ denotes a distance-regular
graph with diameter $D\geq 3$. Recall the valency $k$ of $\Gamma$. 
\begin{lemma} \label{lem:prin}
Fix $x \in X$ and write $T=T(x)$. Then for $0 \leq i,j,\ell \leq D$ and $y,z \in X$ the $(y,z)$-entry of
 $A_i A^*_\ell A_j$ is equal to
\begin{align*}
\vert X \vert \sum_{w \in \Gamma_i(y)\cap \Gamma_j(z)} \langle E_\ell \hat x, E_\ell \hat w \rangle.
\end{align*}
\end{lemma}
\begin{proof} We have
\begin{align*}
&(A_i A_\ell^* A_j)_{y,z} = \sum_{w \in X} (A_i)_{y,w} (A_\ell^*)_{w,w} (A_j)_{w,z} 
= \sum_{w \in \Gamma_i (y) \cap \Gamma_j(z)} (A_\ell^*)_{w,w} \\
&=\vert X \vert \sum_{w \in \Gamma_i (y) \cap \Gamma_j(z)} \langle  E_\ell \hat x, E_\ell \hat w\rangle.
\end{align*}
\end{proof}

\noindent For the rest of this section, let $E$ denote a nontrivial primitive idempotent of $\Gamma$,
and write $E=\vert X \vert^{-1} \sum_{n=0}^D \theta^*_n A_n$. Recall that $E$ is nondegenerate if and only if $\theta^*_n \not=\theta^*_0$ for $1 \leq n \leq D$.
Also recall that for $y,z \in X$,
\begin{align} \label{eq:innerp}
\langle E\hat y, E \hat z\rangle = \vert X \vert^{-1} \theta^*_n, \qquad \qquad n = \partial (y,z).
\end{align}

\begin{theorem} \label{thm:bsc} 
{\rm (See \cite[Theorem~1.1]{QPchar}.)}
With the above notation, the following {\rm (i)--(iii)} are equivalent:
\begin{enumerate}
\item[\rm (i)] $\Gamma$ is $Q$-polynomial with respect to $E$;
\item[\rm (ii)] $E$ is nondegenerate and for all $0 \leq i,j\leq D$ and all distinct $y,z \in X$,
\begin{align*}
\sum_{w \in \Gamma_i(y) \cap \Gamma_j(z)} E {\hat w} -
\sum_{w \in \Gamma_j(y) \cap \Gamma_i(z)} E {\hat w}  = p^h_{i,j} \frac{\theta^*_i -\theta^*_j}{\theta^*_0-\theta^*_h} \bigl(E \hat y - E \hat z\bigr),
\end{align*}
where $h = \partial (y,z)$;
\item[\rm (iii)] $E$ is nondegenerate and for all $y, z\in X$,
\begin{align*}
\sum_{w \in \Gamma(y) \cap \Gamma_2(z)} E {\hat w} -\sum_{w \in \Gamma_2(y) \cap \Gamma(z)} E {\hat w}  \in {\rm Span} (E \hat y - E \hat z).
\end{align*}
\end{enumerate}
\end{theorem}
\begin{proof} Fix $x \in X$ and write $T=T(x)$. Write $E=E_1$ and $A^*=A^*_1$.\\
\noindent ${\rm (i)\Rightarrow (ii)}$ $E$ is nondegenerate since $\lbrace \theta^*_n\rbrace_{n=0}^D$ are mutually distinct.
Recall the Bose-Mesner algebra $M$. By Lemma \ref{lem:step4},
\begin{align*}
{\rm Span}\lbrace R A^* S - S A^* R \vert R, S \in M\rbrace = \lbrace YA^*-A^*Y \vert Y \in M\rbrace.
\end{align*}
Taking $R=A_i$ and $S=A_j$, we obtain
\begin{align}
\label{eq:one}
A_i A^* A_j - A_j A^* A_i = \sum_{n=1}^D r^n_{i,j} \bigl(A^*A_n - A_n A^*\big)
\end{align}
for some scalars $\lbrace r^n_{i,j}\rbrace_{n=1}^D$ in $\mathbb R$.
We show that
\begin{align}
\label{eq:rp1}
r^h_{i,j} = p^h_{i,j} \frac{\theta^*_i - \theta^*_j}{\theta^*_0-\theta^*_h} \qquad \quad (1 \leq h \leq D).
\end{align}
Let $h$ be given, and pick $z \in X$ such that $\partial(x,z)=h$. We compute the $(x,z)$-entry of each term in \eqref{eq:one}. 
We do this using Lemma  \ref{lem:prin} (with $\ell=1$ and $y=x$) along with \eqref{eq:innerp}. A brief calculation yields
\begin{align*}
p^h_{i,j} (\theta^*_i - \theta^*_j) = r^h_{i,j} (\theta^*_0-\theta^*_h),
\end{align*}
and \eqref{eq:rp1} follows. Pick distinct $y,z \in X$ and write $h=\partial(y,z)$. We show that 
\begin{align}
\label{eq:bsgoal}
\sum_{w \in \Gamma_i(y) \cap \Gamma_j(z)} E {\hat w} -
\sum_{w \in \Gamma_j(y) \cap \Gamma_i(z)} E {\hat w}  = p^h_{i,j} \frac{\theta^*_i -\theta^*_j}{\theta^*_0-\theta^*_h} \bigl(E \hat y - E \hat z\bigr).
\end{align}
Since the base vertex $x$ is arbitrary, without loss of generality it suffices to show that in
\eqref{eq:bsgoal}, the left-hand side minus the right-hand side is orthogonal to $E \hat x$.
This orthogonality is routinely obtained from \eqref{eq:one}
and
 \eqref{eq:rp1} along with Lemma \ref{lem:prin} (with $\ell=1$).
\\
\noindent ${\rm (ii)} \Rightarrow {\rm (iii)}$. Clear.
\\
\noindent ${\rm (iii)} \Rightarrow {\rm (i)}$. We assume that $E$ is nondegenerate, so $\theta^*_n \not=\theta^*_0$ for $1 \leq n \leq d$.
\\
\noindent Claim 1. Pick an integer $h$ $(1 \leq h \leq D)$ and $y,z \in X$ such that $\partial (y,z)=h$. Then
\begin{align*}
\sum_{w \in \Gamma(y) \cap \Gamma_2(z)} E {\hat w} -
\sum_{w \in \Gamma_2(y) \cap \Gamma(z)} E {\hat w}  = r^h_{1,2} (E\hat y - E \hat z),
\end{align*}
where
\begin{align}
\label{eq:rp}
r^h_{1,2} = p^h_{1,2} \frac{\theta^*_1 - \theta^*_2}{\theta^*_0 - \theta^*_h}.
\end{align}
\noindent Proof of Claim 1. By assumption there exists $\alpha \in \mathbb R$ such that
\begin{align*}
\sum_{w \in \Gamma(y) \cap \Gamma_2(z)} E {\hat w} -
\sum_{w \in \Gamma_2(y) \cap \Gamma(z)} E {\hat w}  = \alpha(E\hat y - E \hat z).
\end{align*}
For each term in the above equation, take the inner product with $E \hat y$ using \eqref{eq:innerp}. 
A brief calculation yields
\begin{align*}
 p^h_{1,2} (\theta^*_1 - \theta^*_2) = \alpha (\theta^*_0-\theta^*_h).
\end{align*}
Therefore
\begin{align*}
\alpha = p^h_{1,2}\frac{\theta^*_1 - \theta^*_2}{\theta^*_0- \theta^*_h},
\end{align*}
and Claim 1 is proved.
\\
\noindent Claim 2. We have
\begin{align}\label{eq:aasa2}
A A^* A_2 - A_2 A^* A = \sum_{n=1}^D r^n_{1,2}(A^* A_n- A_n A^*).
\end{align}
\noindent Proof of Claim 2. For $y,z \in X$ we compute the $(y,z)$-entry of the left-hand side of \eqref{eq:aasa2} minus
the right-hand side of \eqref{eq:aasa2}. We do this computation using Lemma \ref{lem:prin} (with $\ell=1$). For $y=z$  the $(y,z)$-entry is zero. For $y \not=z$ the $(y,z)$-entry
 is equal to $\vert X \vert$ times
\begin{align*}
\Biggl\langle E \hat x,
\sum_{w \in \Gamma(y) \cap \Gamma_2(z)} E {\hat w} -
\sum_{w \in \Gamma_2(y) \cap \Gamma(z)} E {\hat w}  - r^h_{1,2} (E\hat y - E \hat z) \Biggr \rangle, 
\end{align*}
where $h=\partial(y,z)$. 
The above scalar is zero by Claim 1. Claim 2 is proved.
\\
\noindent Conceivably $\theta^*_1 = \theta^*_2$. In this case $r^h_{1,2}=0$ for $1 \leq h \leq D$. So by Claim 2,
$A A^* A_2 = A_2 A^* A$.
In this equation we eliminate $A_2$ using
 $A_2 = (A^2 - a_1 A - k I)/c_2$
 and get
 \begin{align}
 A^2 A^* A - A A^* A^2 = k(A^* A-A A^*).
 \label{eq:conceivable}
 \end{align}
 We will return to this equation shortly.
 \\
 \noindent Claim 3. Assume that $\theta^*_1 \not=\theta^*_2$. Then there exist scalars $\beta, \gamma, \varrho \in \mathbb R$ such that
 \begin{align} \label{eq:TD1}
 0 = \lbrack A, A^2 A^* - \beta A A^*A+ A^* A^2 - \gamma(AA^*+A^*A) - \varrho A^*\rbrack.
 \end{align}
 \noindent Proof of Claim 3. Referring to  \eqref{eq:rp}, the scalar $p^h_{1,2}$ is zero if $h>3$ and nonzero if $h=3$. Therefore
 $r^h_{1,2}$ is zero if $h>3$ and nonzero if $h=3$. The matrices $A_2$ and $A_3$ appear in  \eqref{eq:aasa2}.
 Recall that $A_2$ and $A_3$  are polynomials in $A$ that have degrees 2 and 3, respectively. Evaluating \eqref{eq:aasa2} using
 this fact,  we obtain
 \begin{align*}
 A^3 A^*-A^* A^3 \in {\rm Span} \Bigl ( A^2 A^* A-A A^* A^2, A^2 A^*-A^* A^2, AA^*-A^*A\Bigr).
 \end{align*}
 Therefore there exist $\beta, \gamma, \varrho \in \mathbb R$ such that
 \begin{align*}
  A^3 A^*-A^*A^3 = (\beta+1)(A^2 A^*A - A A^* A^2) +\gamma(A^2 A^*-A^*A^2) +\varrho (AA^*-A^*A).
 \end{align*}
 In this equation we rearrange the terms to obtain \eqref{eq:TD1}. Claim 3 is proved. \\
 Recall our notation $E=E_1$. For the time being, let $\lbrace E_i \rbrace_{i=2}^D$ denote any ordering
of the remaining nontrivial primitive idempotents of $\Gamma$. For $0 \leq i \leq D$ let  $\theta_i$ denote the eigenvalue of $\Gamma$ for $E_i$.
 Recall the graph $\Delta_E$ from Definition \ref{def:delta}. The graph $\Delta_E$ is connected since $E$ is nondegenerate.
 Recall that in $\Delta_E$, vertex 0 is adjacent to vertex 1 and no other vertex. We will show that $\Delta_E$ is a path.
 To do this, it suffices to show that each vertex $i$ in $\Delta_E$ is adjacent to at most 2 vertices in $\Delta_E$.
 \\
 \noindent Claim 4. For distinct vertices $i,j$ in $\Delta_E$ that are adjacent,
 \begin{enumerate}
 \item[\rm (i)] if $\theta^*_1 = \theta^*_2$ then $\theta_i \theta_j = -k$;
 \item[\rm (ii)] if $\theta^*_1 \not=\theta^*_2$ then $P(\theta_i, \theta_j)= 0$,
 where we recall
 \begin{align*}
 P(\lambda, \mu) = \lambda^2 - \beta \lambda \mu + \mu^2 - \gamma (\lambda + \mu) - \varrho.
 \end{align*}
 \end{enumerate}
 \noindent Proof of Claim 4. First assume that $\theta^*_1 = \theta^*_2$. Then  \eqref{eq:conceivable} holds. In  \eqref{eq:conceivable}, multiply each term on the left by $E_i$
and on the right by $E_j$. Simplify the result to get
\begin{align*}
0 = E_i A^* E_j (\theta_i - \theta_j) (\theta_i \theta_j+k).
\end{align*}
We have $E_i A^* E_j \not=0$ since $i,j$ are adjacent in $\Delta_E$. The scalar $\theta_i - \theta_j$ is nonzero since $i \not=j$. Therefore
$\theta_i \theta_j+k=0$ so $\theta_i \theta_j = -k$.
Next assume that $\theta^*_1 \not=\theta^*_2$. Then \eqref{eq:TD1} holds. In \eqref{eq:TD1},
 multiply each term on the left by $E_i$ and the right by $E_j$. Simplify the result to get
 \begin{align*}
 0 = E_i A^* E_j (\theta_i - \theta_j) P(\theta_i, \theta_j).
 \end{align*}
We have $E_i A^* E_j \not=0$ since $i,j$ are adjacent in $\Delta_E$. The scalar $\theta_i - \theta_j$ is nonzero since $i \not=j$. Therefore
$P(\theta_i,\theta_j)=0$.
\\
\noindent Claim 5. We have $\theta^*_1 \not=\theta^*_2$.
\\
\noindent Proof of Claim 5. Suppose that $\theta^*_1 = \theta^*_2$. By Claim 4 and since vertex 0 is adjacent to vertex 1, we have $\theta_0 \theta_1 = -k$.
We have $\theta_0=k$ so $\theta_1 =-1$. The graph $\Delta_E$ is connected, so vertex 1 is adjacent to some nonzero vertex $j$. By Claim 4 we have
$\theta_1 \theta_j = -k$. By this and $\theta_1=-1$, we obtain $\theta_j=k$. This implies $j=0$, for a contradiction. Claim 5 is proved.
\\
\noindent Claim 6. Each vertex $i$ in $\Delta_E$ is adjacent at most two vertices in $\Delta_E$.
\\
\noindent Proof of Claim 6. By Claims 4, 5 we see that for each vertex $j$ in $\Delta_E$ that is adjacent vertex $i$, the eigenvalue $\theta_j$ is a root of
the polynomial
\begin{align*}
P(\theta_i, \mu) = \theta^2_i - \beta \theta_i \mu + \mu^2 -\gamma(\theta_i + \mu) - \varrho.
\end{align*}
This polynomial is quadratic in $\mu$, so it has at most two distinct roots. Claim 6 is proved.
\\
\noindent We have shown that the graph $\Delta_E$ is a path. Consequently the graph  $\Gamma$ is $Q$-polynomial with respect to $E$.
\end{proof}

\noindent The  balanced set condition has subtle combinatorial implications; see
\cite{caugh3, caugh4, ckt, lewis1, stefkoUP, 4vert, terwSub3}.


\section{Directions for future research}
\noindent In this section, we extend the $Q$-polynomial property to graphs that are not necessarily distance-regular.
\medskip

\noindent Throughout this section, let $\Gamma=(X,\mathcal R)$ denote a finite, undirected, connected graph, without loops or multiple edges, with diameter $D\geq 1$.
We do not assume that $\Gamma$ is distance-regular. Let $A$ denote the adjacency matrix of $\Gamma$. 

\begin{definition}\rm An ordering $\lbrace V_i \rbrace_{i=0}^d$ of the eigenspaces of $A$ is called {\it $Q$-polynomial} whenever 
\begin{align} \label{eq:QPorder}
\sum_{\ell=0}^i V_\ell = \sum_{\ell=0}^i (V_1)^{\circ \ell} \qquad \qquad (0 \leq i \leq d).
\end{align}
We are using the notation \eqref{eq:copies}.
\end{definition}

\begin{definition}\rm The graph $\Gamma$ is said to be {\it $Q$-polynomial} whenever there exists at least one $Q$-polynomial ordering
of the eigenspaces of $A$.
\end{definition}

\begin{lemma} \label{lem:reg} Assume that $\Gamma$ is $Q$-polynomial. Then $\Gamma$ is regular.
\end{lemma} 
\begin{proof} Let $\lbrace V_i \rbrace_{i=0}^d$ denote a $Q$-polynomial ordering of the eigenspaces of $A$. Setting $i=0$ in \eqref{eq:QPorder},
we obtain $V_0 = \mathbb R {\bf 1}$. Therefore, the vector $\bf 1$ is an eigenvector for $A$. Consequently $\Gamma$ is regular.
\end{proof}
\noindent Assume for the moment that $\Gamma$ is $Q$-polynomial. We do not expect that $\Gamma$ is distance-regular. However, we do expect that
the following conjecture is true. Let $M$ denote the subalgebra of ${\rm Mat}_X(\mathbb R)$ generated by $A$.

\begin{conjecture}\rm \label{conj:BC} If $\Gamma$ is $Q$-polynomial, then $B\circ C \in M$ for all $B, C \in M$.
\end{conjecture}
\begin{note}\rm Conjecture \ref{conj:BC} asserts that if $\Gamma$ is $Q$-polynomial then $M$ is the Bose-Mesner algebra
of a symmetric association scheme \cite[Section~2.2]{bcn}.
\end{note}

\noindent For the rest of this section, fix $x \in X$. Define $D(x) = {\rm max} \lbrace \partial(x,y)\vert y \in X\rbrace$. 
For $0 \leq i \leq D(x) $ define the matrix $E^*_i = E^*_i(x)$ as in line \eqref{DEFDEI}. Note that $\lbrace E^*_i\rbrace_{i=0}^{D(x)} $ form a basis
for a commutative subalgebra $M^*=M^*(x)$ of ${\rm Mat}_X(\mathbb R)$.

\begin{definition}\rm Let $\lbrace V_i \rbrace_{i=0}^d$ denote an ordering of the eigenspaces of $A$. A matrix $A^* \in {\rm Mat}_X(\mathbb R)$
is called a {\it dual adjacency matrix} (with respect to $x$ and $\lbrace V_i \rbrace_{i=0}^d$) whenever:
\begin{enumerate}
\item[\rm (i)] $A^*$ generates $M^*$;
\item[\rm (ii)] for $0 \leq i \leq d$ we have
\begin{align*}
A^* V_i \subseteq V_{i-1} + V_i + V_{i+1},
\end{align*}
where $V_{-1}=0$ and $V_{d+1}=0$.
\end{enumerate}
\end{definition}
\begin{definition}\rm An ordering $\lbrace V_i \rbrace_{i=0}^d$ of the eigenspaces of $A$ is called
{\it  $Q$-polynomial with respect to $x$} whenever there exists a dual adjacency matrix with respect to $x$ and
 $\lbrace V_i \rbrace_{i=0}^d$.
 \end{definition}

\begin{definition}\rm \label{def:Qx} \rm We say that $\Gamma$ is {\it $Q$-polynomial with respect to $x$} whenever there exists an ordering of the eigenspaces of $A$
that is $Q$-polynomial with respect to $x$.
\end{definition}
 
 \noindent Another generalization we could adopt, is to allow  the adjacency matrix $A$ to be  weighted. A weighted adjacency
 matrix is obtained from the classical adjacency matrix by replacing each entry 1 by a nonzero real scalar.
 The only requirement on the scalars is that the weighted adjacency matrix is diagonalizable.
 A weighted adjacency matrix is used in \cite{curtin6}. 
 \medskip
 
\begin{problem}\rm Investigate how the above variations on the $Q$-polynomial property are related.
\end{problem} 

 \begin{remark} \rm In \cite{suda} Sho Suda introduced the $Q$-polynomial property for coherent configurations. A coherent configuration
 is a combinatorial object more general than a graph.
 \end{remark}
\section{Acknowledgement} The author thanks Edwin van Dam, Tatsuro Ito,  Jack Koolen, and Hajime Tanaka for helpful comments about the paper.

\bigskip

\noindent Paul Terwilliger \hfil\break
\noindent Department of Mathematics \hfil\break
\noindent University of Wisconsin \hfil\break
\noindent 480 Lincoln Drive \hfil\break
\noindent Madison, WI 53706-1388 USA \hfil\break
\noindent email: {\tt terwilli@math.wisc.edu }\hfil\break

\end{document}